\documentclass[a4paper]{article}

\usepackage{amsfonts}
\usepackage{amssymb}
\usepackage{amsthm}
\usepackage{amsmath}

\title{Freiman's theorem in an arbitrary nilpotent group}
\date{}
\author{Matthew Tointon}

\numberwithin{equation}{section}

\newtheorem{prop}{Proposition}[section]
\newtheorem{theorem}[prop]{Theorem}
\newtheorem{lemma}[prop]{Lemma}
\newtheorem{corollary}[prop]{Corollary}

\theoremstyle{definition}
\newtheorem{definition}[prop]{Definition}

\theoremstyle{remark}
\newtheorem{remark}[prop]{Remark}

\newtheorem{example}[prop]{Example}
\newtheorem*{remark*}{Remark}
\newtheorem*{remarks*}{Remarks}

\newcommand{\N}{\mathbb{N}}
\newcommand{\Z}{\mathbb{Z}}

\newcommand{\F}{\mathbb{F}}

\newcommand{\rank}{\text{rank }}

\begin{document}
\maketitle
\begin{center}
\scriptsize{\textit{Department of Pure Mathematics and Mathematical Statistics, Centre for Mathematical Sciences, University of Cambridge, Wilberforce Road, Cambridge CB3 0WB, United Kingdom}}\\
\scriptsize{\textit{email: M.Tointon@dpmms.cam.ac.uk}}
\end{center}
\begin{abstract}We prove a Freiman--Ruzsa-type theorem valid in an arbitrary nilpotent group. Specifically, we show that a $K$-approximate group $A$ in an $s$-step nilpotent group $G$ is contained in a coset nilprogression of rank at most $K^{O_s(1)}$ and cardinality at most $\exp(K^{O_s(1)})|A|$. To motivate this, we give a direct proof of Breuillard and Green's analogous result for torsion-free nilpotent groups, avoiding the use of Mal'cev's embedding theorem.
\end{abstract}
\tableofcontents
\section{Introduction}
In various situations in mathematics, considering approximate analogues of precise algebraic properties can lead to the objects they define becoming far more widely applicable. An area that has received considerable attention in that regard in recent years is the study of \emph{approximate groups}. Originating in additive combinatorics, these objects have increasingly been applied to obtain results in a diverse array of other fields, such as number theory, differential geometry and theoretical computer science; see, for example, \cite{arb.group,app.grps} and the references therein for details on some of these applications.

Approximate group theory really began with the celebrated theorem of Freiman \cite{freiman} describing the structure of sets of integers that are approximately closed under addition. Specifically, writing
\[
A+B:=\{a+b:a\in A,b\in B\},
\]
Freiman proved that if $A$ is a finite subset of $\Z$ satisfying
\begin{equation}\label{eq:sm.doub}
|A+A|\le K|A|
\end{equation}
then $A$ is contained in a \emph{progression} $P$ of \emph{rank} $r\ll_K1$ and cardinality $|P|\ll_K|A|$, which is to say a set of the form
\begin{equation}\label{eq:prog}
P=\{x_0+l_1x_1+\ldots+l_rx_r:0\le l_i\le L_i\}.
\end{equation}
Ruzsa \cite{ruzsa.Z} subsequently gave a simplified proof of this result.

In a more general abelian group $G$ it still makes sense to ask for a classification of sets satisfying (\ref{eq:sm.doub}). However, Freiman's theorem cannot hold exactly as stated for $\Z$, since if $G$ is a finite abelian group of sufficiently high rank then $G$ itself satisfies (\ref{eq:sm.doub}) for every $K\ge1$, but it cannot be contained in any set $P$ of the form (\ref{eq:prog}) with $r\ll1$.

Nonetheless, Green and Ruzsa \cite{green-ruzsa} showed that this is essentially the only way in which Freiman's theorem can fail in an abelian group. Indeed, they showed that if $G$ is an arbitrary abelian group and $A\subset G$ satisfies (\ref{eq:sm.doub}) then $A$ is contained inside a set $H+P$, where $H$ is a subgroup, $P$ is of the form (\ref{eq:prog}) and $|H+P|\ll_K|A|$. A set of this form is called a \emph{coset progression}. Their result is stated precisely, with explicit bounds, as Theorem~\ref{thm:ab.freiman} in the present work. The best bounds currently available for a result of this type in an arbitrary abelian group are due to Sanders \cite{sanders}.

It is natural to ask whether similar results hold in non-abelian groups. To do so requires some more notation. If $A$ and $B$ are subsets of a multiplicative group then we denote by $AB$ the set $\{ab:a\in A,b\in B\}$ and by $A^{-1}$ the set $\{a^{-1}:a\in A\}$. We define $A^1:=A$ and $A^n:=A(A^{n-1})$. Similarly, if $A$ is a subset of an (additive) abelian group then we define $-A:=\{-a:a\in A\}$, $1A:=A$ and $nA:=A+(n-1)A$.

The most na\"ive analogue of condition (\ref{eq:sm.doub}) for a non-abelian group would be the condition $|AA|\le K|A|$. However, it turns out (see \cite{tao.product.set} for details) that the appropriate analogue in the non-abelian setting is to require that $A$ is an \emph{approximate group}, defined as follows.
\begin{definition}[Approximate group]\label{def:app.grp}
Let $G$ be a (multiplicative) group and let $K\ge1$. A finite set $A\subset G$ is called a \emph{$K$-approximate (sub)group} if
\begin{enumerate}
\renewcommand{\labelenumi}{(\roman{enumi})}
\item it is symmetric, which is to say that $A^{-1}=A$, and contains the identity;
\item there is a symmetric subset $X$ with $|X|\le K$ such that $A^2\subset XA$.
\end{enumerate}
If $A$ generates a nilpotent subgroup of step $s$ then we say that $A$ is an \emph{$s$-step $K$-approximate group}.
\end{definition}
\begin{remark}We may assume, without loss of generality, that $X\subset A^3$. Indeed, we may certainly assume that $X$ is a minimal symmetric set satisfying $A^2\subset XA$, and in that case for each $x\in X$ there exist $a_1,a_2,a_3$ such that $a_1a_2=xa_3$, and so $x=a_1a_2a_3^{-1}\in A^3$.
\end{remark}
\begin{remark}
When considering a $K$-approximate group we may (and do) assume that $K$ is an integer. Moreover, we may (and do) assume that $K\ge2$, since if $K=1$ then $A$ is a genuine finite subgroup and our results are trivial. One consequence of this that is convenient from a notational perspective is that if $p$ is a polynomial then $p(K)\le K^{O_p(1)}$, and so, for example, bounds of the form $O(K^{O(1)})$ may be rewritten as simply $K^{O(1)}$.
\end{remark}
Remarkably, Breuillard, Green and Tao \cite{arb.group} have shown that, with a suitable modification of the notion of a progression, if $A$ is a $K$-approximate group inside an arbitrary group\footnote{Their result is also valid in the more general setting of \emph{local} groups.} then $A$ can still be described in terms of a subgroup and a progression. Specifically, they show that $A^4$ contains a set $HP$ satisfying $|HP|\gg_K|A|$, where $H$ is a normal subgroup of a subgroup $G_0$ of $G$ and the image of $P$ in the quotient $G_0/H$ is a \emph{nilprogression} of rank and step $O_K(1)$, defined as follows.
\begin{definition}[Nilprogression]\label{def:nilprog}
Let $x_1,\ldots,x_r$ be elements that generate an $s$-step nilpotent group and let $L=(L_1,\ldots,L_r)$ be a vector of positive integers. Then the set of all products in the $x_i$ and their inverses in which each $x_i$ and its inverse appear at most $L_i$ times between them is called a \emph{nilprogression} of rank $r$, step $s$ and side lengths $L_1,\ldots,L_r$, and is denoted $P^*(x_1,\ldots,x_r;L)$.
\end{definition}
The proof of this result is long and far from elementary, and the bounds it gives are ineffective, and so it is of interest to have shorter proofs and better bounds for particular classes of $G$.

A fair amount is known in the case in which $G$ is a nilpotent group. For example, an old argument of Ruzsa \cite{ruzsa.torsion} can easily be adapted to show that if $G$ is $s$-step nilpotent and every element of $G$ has order at most $m$ then a $K$-approximate group $A$ must be contained inside a genuine subgroup of cardinality at most $O_{K,m,s}(|A|)$. On the other hand, Breuillard and Green \cite{tor.free.nilp} have shown that if $A$ is a subset of a torsion-free nilpotent group then $A$ is `controlled'\footnote{The notion of control is standard in approximate group theory and gives a way of describing one set in terms of another. See \cite{tor.free.nilp,tao.product.set} for details.} by a set $P(x_1,\ldots,x_r;L)$ called a \emph{nilpotent progression}, which we define in Definition \ref{def:nilp.prog}. The notion of a nilpotent progression is not equivalent to that of a nilprogression, but for the purposes of this paper they can be thought of as roughly equivalent; see Appendix \ref{sec:rough.equiv.prog} for details.

When a nilpotent group $G$ contains elements of finite but unbounded order, however, the question has remained open. The present work resolves this question. Specifically, we prove the following result.
\begin{theorem}\label{thm:freiman.p-group}
Let $G$ be an $s$-step nilpotent group. Suppose that $A\subset G$ is a $K$-approximate group. Then there exist a subgroup $H$ of $G$, normalised by $A$; an integer $k\le K^{O_s(1)}$; elements $x_1,\ldots,x_k$ of $G$; and lengths $L=(L_1,\ldots,L_k)$ such that the nilprogression $P^*=P^*(x_1,\ldots,x_k;L)$ and the nilpotent progression $P=P(x_1,\ldots,x_k;L)$ satisfy
\[
A\subset HP^*\subset HP\subset A^{K^{O_s(1)}}.
\]
\end{theorem}
\begin{remark}
In particular, by Lemma~\ref{lem:app.grp.fundamental}, $|HP^*|\le|HP|\le\exp(K^{O_s(1)})|A|$.
\end{remark}
An immediate corollary of Theorem \ref{thm:freiman.p-group} is the following refinement of the Breuillard--Green theorem.
\begin{theorem}\label{thm:freiman.tor.free.nilp}
Let $G$ be torsion-free $s$-step nilpotent group. Suppose that $A\subset G$ is a $K$-approximate group. Then there exist $k\le K^{O_s(1)}$, elements $x_1,\ldots,x_k$ of $G$ and lengths $L=(L_1,\ldots,L_k)$ such that the nilprogression $P^*=P^*(x_1,\ldots,x_k;L)$ and the nilpotent progression $P=P(x_1,\ldots,x_k;L)$ satisfy
\[
A\subset P^*\subset P\subset A^{K^{O_s(1)}}.
\]
\end{theorem}
\begin{remark}
In particular, by Lemma~\ref{lem:app.grp.fundamental}, $|P^*|\le|P|\le\exp(K^{O_s(1)})|A|$.
\end{remark}
\begin{remark}It follows from \cite{nilp.dim.lemma} that the group generated by $A$ in Theorem \ref{thm:freiman.tor.free.nilp} is of nilpotency class at most $O_K(1)$, and so the bounds in Theorem \ref{thm:freiman.tor.free.nilp} can be taken to be independent of $s$ (though they will no longer be polynomial in $K$).
\end{remark}
\begin{remark}In principle it should be possible to reduce the rank $k$ of the nilprogression in Theorems \ref{thm:freiman.p-group} and \ref{thm:freiman.tor.free.nilp} to $6\log_2K$, as in \cite[Theorem 2.12]{arb.group}, although this would be at the expense of having $A$ covered by $O_K(1)$ translates of $HP^*$ rather than just one. We do not pursue this matter in the present work.
\end{remark}
Many of the arguments we use to prove Theorem \ref{thm:freiman.p-group} are considerably simpler under the assumption that $G$ is torsion free, and in that case they yield a very different proof of Theorem \ref{thm:freiman.tor.free.nilp} from the argument of Breuillard and Green. Therefore, although Theorem \ref{thm:freiman.tor.free.nilp} can be deduced directly from Theorem \ref{thm:freiman.p-group}, before proving Theorem \ref{thm:freiman.p-group} we present a direct proof of Theorem \ref{thm:freiman.tor.free.nilp}. This will motivate much of what comes later, as well as making the arguments easier to digest initially.

The starting point for the proofs of Theorems \ref{thm:freiman.p-group} and \ref{thm:freiman.tor.free.nilp} is the so-called \emph{splitting lemma} of Tao \cite[Lemma 7.7]{tao.product.set}, which we present in Section~\ref{sec:tools}. If an approximate group $A$ lies inside a group $G$ with normal subgroup $N$, this lemma essentially allows one to express $A$ in terms of an approximate subgroup $C$ of $G/N$ and an approximate subgroup $B$ of $N$.

As was pointed out in \cite{tao.product.set}, this suggests an inductive approach to the study of nilpotent approximate groups, since in the case of a nilpotent group $G$, taking $N$ to be the commutator subgroup of $G$ would mean that both $B$ and $C$ were nilpotent of lower step. However, a difficulty in applying the splitting lemma, also identified in \cite{tao.product.set}, is that it gives $A\subset\phi(C)B$, where $\phi$ is a certain `approximate homomorphism' from $G/N$ to $G$, and such approximate homomorphisms seem to be difficult to classify.

With this in mind, the principal combinatorial idea contained in the present work is to be found in Proposition \ref{prop:prod.of.approx.grps.tor.free}. In this proposition, in the torsion-free case we are able to use the properties of $\phi$, as well as a converse to the splitting lemma that is also found in \cite{tao.product.set}, to express $A$ in terms of a product of boundedly many approximate groups of lower step, without ever having to describe $\phi$ in detail. A fairly standard application of an argument of Chang then allows us to place $A$ inside a product $A_1\cdots A_r$ of $K^{O(1)}$-approximate groups of lower step, in which $A_i\subset A^{O(1)}$ and $r\le K^{O(1)}$.

Once we have this statement the problem basically becomes an algebraic one. In particular, in the torsion-free case it is straightforward to proceed by induction to a situation in which the $A_i$ are all abelian. The Green--Ruzsa theorem then allows us to assume that $A$ is contained inside a product of abelian progressions, and it is then a conceptually easy, if technically tedious, matter to place $A$ inside a nilpotent progression. The details are given in Section~\ref{sec:ind.tor.free} and Appendix \ref{sec:rough.equiv.prog}.

Once we introduce the possibility of torsion into the problem the need to deal with finite subgroups makes things somewhat trickier. In Section \ref{sec:strategy} we describe our strategy for dealing with the resultant issues, before concluding the argument in Sections \ref{sec:img.of.multi.hom}-\ref{sec:conclusion}.

The arguments of Section \ref{sec:img.of.multi.hom} in particular rely on certain fundamental properties of $p$-groups; we present these properties in Appendix \ref{sec:p.groups}.
\subsection*{Acknowledgements}While this work was carried out the author was supported by an EPSRC doctoral training grant, awarded by the Department of Pure Mathematics and Mathematical Statistics in Cambridge. For much of that time he was also supported by a Bye-Fellowship from Magdalene College, Cambridge, and he is grateful to the college for their generosity. 

It is a pleasure to thank Ben Green for suggesting the problem and for numerous helpful conversations since, and Emmanuel Breuillard, Tom Sanders and an anonymous referee for a number of helpful comments on previous versions of this paper.
\section{Tools from the literature}\label{sec:tools}
In this section we collect together some tools from the literature on approximate groups. Let us first, however, record the following trivial but repeatedly useful observation.
\begin{lemma}\label{lem:app.grp.fundamental}
If $A$ is a $K$-approximate group then $|A^n|\le K^{n-1}|A|$ for every positive integer $n$.
\end{lemma}
We now review some elements of abelian approximate group theory. Given a set $A$ in an abelian group, define the \emph{doubling constant} of $A$ to be the quantity $|A+A|/|A|$. The key abelian result is the following generalisation by Green and Ruzsa of Freiman's theorem.
\begin{sloppypar}
\begin{theorem}[Green--Ruzsa \cite{green-ruzsa}]\label{thm:green-ruzsa}Suppose that $A$ is a symmetric subset of an abelian group with doubling constant at most $K$. Then $4A$ contains a coset progression $H+P$ of rank at most $O(K^{O(1)})$ and cardinality at least $\exp(-O(K^{O(1)}))|A|$.
\end{theorem}
\end{sloppypar}
\begin{remark}We assume symmetry here, and throughout this section, for purely notational reasons, so as to avoid the need to distinguish between $A$ and $-A$.
\end{remark}
The following covering argument of Chang \cite{chang} shows that the approximate group $A$ is covered by a few translates of the coset progression $H+P$ given by Theorem~\ref{thm:green-ruzsa}.
\begin{prop}[Chang]\label{prop:chang}
Let $G$ be a group and suppose that $A\subset G$ is a $K$-approximate subgroup. Let $B\subset A^{M}$ be a set with $|B|\ge|A|/M'$. Then there exist sets $S_1,\ldots,S_t\subset A$ with $t\ll\log M'+M\log K$ such that $|S_i|\le2K$ and
\begin{equation}\label{eq:chang.0}
A\subset S_{t-1}^{-1}\cdots S_1^{-1}B^{-1}BS_1\cdots S_t.
\end{equation}
\end{prop}
\begin{proof}This is essentially identical to the argument given in \cite{chang}, with the notation abstracted to allow for applications in non-abelian groups. Set $B_0:=B$. Now, given $B_i$, let $R_{i+1}$ be a maximal subset of $A$ with respect to the property that the translates $B_ix$ with $x\in R_{i+1}$ are disjoint. If $|R_{i+1}|>2K$ then let $S_{i+1}$ be an arbitrary subset of $R_{i+1}$ of cardinality $2K$ and set $B_{i+1}:=B_iS_{i+1}$. If $|R_{i+1}|\le2K$ then set $S_{i+1}=R_{i+1}$ and stop.

The fact that the translates $B_ix$ with $x\in S_{i+1}$ are disjoint implies that $|B_{i+1}|=|B_i||S_{i+1}|$, which in turn implies that that
\begin{equation}\label{eq:chang.1}
|B_i|=|B||S_1|\cdots|S_i|\ge (2K)^i|A|/M'.
\end{equation}
On the other hand, we have
\[
B_i\subset BS_1\cdots S_i\subset A^{M+i},
\]
and so
\begin{equation}\label{eq:chang.2}
|B_i|\le K^{M+i-1}|A|
\end{equation}
by Lemma \ref{lem:app.grp.fundamental}. Combining (\ref{eq:chang.1}) and (\ref{eq:chang.2}) we see that
\[
2^iK^i\le M'K^{M+i-1},
\]
which in turn implies that $i\ll\log M'+M\log K$. The algorithm therefore terminates after at most $O(\log M'+M\log K)$ steps, and so we have defined $S_1,\ldots,S_t$ with $|S_i|\le2K$ and $t\ll\log M'+M\log K$, as required. To prove (\ref{eq:chang.0}), note that $S_t\subset A$ is maximal with the property that the translates $B_{t-1}x$ with $x\in S_t$ are disjoint, where of course $B_{t-1}=BS_1\cdots S_{t-1}$. This implies in particular that for every $a\in A$ there exists $x\in S_t$ such that $B_{t-1}x\cap B_{t-1}a\ne\varnothing$, and so $A\subset B_{t-1}^{-1}B_{t-1}S_t$, which is precisely (\ref{eq:chang.0}).
\end{proof}
By viewing the sets $S_i$ given by Chang's covering argument as extra dimensions in the progression $P$ given by Theorem~\ref{thm:green-ruzsa}, we may in fact assume that $A$ is \emph{contained} in $H+P$. The details of that argument are given in \cite{green-ruzsa}. The resulting statement is as follows.
\begin{theorem}[Green--Ruzsa \cite{green-ruzsa}]\label{thm:ab.freiman}Suppose that $A$ is a symmetric subset of an abelian group with doubling constant at most $K$. Then there is a subgroup $H\subset4A$ and a progression $P=P(x_1,\ldots,x_k;L)$ of rank at most $O(K^{O(1)})$ such that
\[
x_i\in4A\qquad(i=1,\ldots,k)
\]
and
\[
A\subset H+P\subset O(K^{O(1)})A.
\]
\end{theorem}
We now move on to the more general theory of approximate groups, beginning with the following straightforward result.
\begin{lemma}\label{lem:pullback}
Let $G$ be a group and let $\rho:G\to H$ be a homomorphism with finite kernel. If $A\subset H$ is a $K$-approximate group then the pullback $\rho^{-1}(A)$ is a $2K$-approximate group.
\end{lemma}
\begin{proof}The set $\rho^{-1}(A)$ is certainly symmetric and contains the identity, and the fact that $\ker\rho$ is finite implies that $\rho^{-1}(A)$ is finite. It remains to check the existence of a symmetric set $\tilde{X}$ of cardinality at most $2K$ such that $\rho^{-1}(A)^2\subset\tilde{X}\rho^{-1}(A)$.

By definition there is a symmetric set $X$ of cardinality at most $K$ such that $A^2\subset XA$. For each $x\in X$ select an element $\omega(x)\in\rho^{-1}(x)$. Set $\tilde X=\{\omega(x):x\in X\}\cup\{\omega(x)^{-1}:x\in X\}$, so that $\tilde X$ is symmetric and of cardinality at most $2K$. Now given $a_1,a_2\in\rho^{-1}(A)$, note that by definition of $X$ there exist $x\in X$ and $a\in\rho^{-1}(A)$ such that $\rho(a_1)\rho(a_2)=x\rho(a)$. This implies that there exists $b\in\ker\rho$ such that $a_1a_2=\omega(x)ab$. However, $ab\in\rho^{-1}(A)$, and so $a_1a_2\in\tilde X\rho^{-1}(A)$, as desired.
\end{proof}
\begin{remark}If we were to assume additionally that $H$ had no elements of order 2 then we could insist that $\omega(x^{-1})=\omega(x)^{-1}$ so that $\{\omega(x):x\in X\}$ was symmetric, and hence conclude that $\rho^{-1}(A)$ was a $K$-approximate group.
\end{remark}
\begin{remark}\label{rem:image.app.grp}It is also trivially the case that if $\tilde A$ is a $K$-approximate subgroup of $G$ then $\rho(\tilde A)$ is a $K$-approximate subgroup of $H$.
\end{remark}
\begin{lemma}[{\cite[\S3]{tao.product.set}}]\label{lem:sm.trip.to.approx.grp}
Let $A$ be a symmetric subset of a group, and suppose that $A$ contains the identity and satisfies $|A^3|\le K|A|$. Then $A^3$ is an $O(K^{O(1)})$-approximate group.
\end{lemma}
\begin{lemma}\label{lem:approx.grp.in.[g,g]}
Let $G$ be a group with subgroup $H$ and suppose that $A\subset G$ is a $K$-approximate subgroup. Then for each $m\ge2$
\begin{enumerate}
\renewcommand{\labelenumi}{(\roman{enumi})}
\item $A^m\cap H$ is a $2K^{2m-1}$-approximate subgroup; and
\item $A^m\cap H$ can be covered by at most $K^{m-1}$ left-translates of $A^2\cap H$.
\end{enumerate}
\end{lemma}
\begin{proof}This is essentially found in \cite[Lemma 10.3]{arb.group}. Since $A^2\cap H\subset A^r\cap H$ and $(A^r\cap H)^2\subset A^{2r}\cap H$, applying (ii) with $m=2r$ would imply that $(A^r\cap H)^2$ could be covered by $K^{2r-1}$ left-translates of $A^r\cap H$. More explicitly, there would be a set $Z$ of cardinality at most $K^{2r-1}$ such that $(A^r\cap H)^2\subset Z(A^r\cap H)$. In particular, $X:=Z\cup Z^{-1}$ would be a symmetric set of cardinality at most $2K^{2r-1}$ such that $(A^r\cap H)^2\subset X(A^r\cap H)$. This would be sufficient to satisfy condition (i) with $m=r$. It therefore suffices to prove (ii) for arbitrary $m$.

The fact that $A$ is a $K$-approximate group implies that there is a subset $Y$ of cardinality at most $K^{m-1}$ such that $A^m\subset YA$, and in particular that
\begin{equation}\label{eq:slicing.1}
A^m\cap H\subset YA.
\end{equation}
We may assume that $Y$ is minimal such that (\ref{eq:slicing.1}) holds, and hence that for each element $y\in Y$ there exists an element
\begin{equation}\label{eq:slicing.2}
\nu(y)\in H\cap yA.
\end{equation}
Fix an arbitrary $y\in Y$ and note that (\ref{eq:slicing.2}) implies that there exist $a\in A$ such that
\begin{equation}\label{eq:slicing.3}
\nu(y)=ya.
\end{equation}
Now note that for arbitrary $a'\in A$ we have $ya'=\nu(y)a^{-1}a'$ by (\ref{eq:slicing.3}), and so
\begin{equation}\label{eq:slicing.4}
yA\subset\nu(y)A^2.
\end{equation}
Moreover, (\ref{eq:slicing.2}) implies that $\nu(y)\in H$, which certainly implies that $H\subset\nu(y)H$. Combined with (\ref{eq:slicing.4}) this gives
\begin{equation}\label{eq:slicing.5}
yA\cap H\subset\nu(y)(A^2\cap H).
\end{equation}
Since $y$ was arbitrary, combining (\ref{eq:slicing.1}) and (\ref{eq:slicing.5}) yields,
\[
A^m\cap H\subset\bigcup_{y\in Y}\nu(y)(A^2\cap H),
\]
from which the result is immediate.
\end{proof}
The rough strategy of the present work is to induct on the step $s$ of $G$ by splitting $G$ into $[G,G]$ and $G/[G,G]$, each of which has step less than $s$. The behaviour of $A$ under this splitting is described by Lemma \ref{lem:splitting} below. Before we state the lemma, let us note a standard definition.
\begin{definition}[Right inverse]
Suppose $\pi:X\to Y$ is some function. Then a function $\phi:\pi(X)\to X$ is called a \emph{right inverse} to $\pi$ if $\pi\circ\phi$ is equal to the identity on $\pi(X)$.
\end{definition}
\begin{lemma}[Tao's splitting lemma]\label{lem:splitting}Let $G$ be a group and let $H$ be a normal subgroup. Write $\pi:G\to G/H$ for the canonical homomorphism. Suppose that $A\subset G$ is a finite symmetric set. Let $m\in\N$. Then there exists a right inverse $\phi:G/H\to G$ to $\pi$ satisfying the following conditions.
\begin{enumerate}
\renewcommand{\labelenumi}{(\roman{enumi})}
\item For every $r\in\N$ we have $\phi(\pi(A^r))\subset A^r$.
\item We have $\phi(1)=1$.
\item For every $x_1,\ldots,x_n\in\pi(A^r)$ with $x_1\cdots x_n=1$, and every $\varepsilon_1,\ldots,\varepsilon_n=\pm1$, we have $\phi(x_1^{\varepsilon_1})^{\varepsilon_1}\cdots\phi(x_n^{\varepsilon_n})^{\varepsilon_n}\in A^{rn}\cap H$ for every $r\in\N$.
\item We have the containment $A\subset\phi(\pi(A))(A^2\cap H)$.
\end{enumerate}
\end{lemma}
\begin{remark}\label{rem:after.splitting}In light of Remark \ref{rem:image.app.grp} and Lemma \ref{lem:approx.grp.in.[g,g]}, in the case that $A$ is a $K$-approximate group the splitting lemma gives a $K$-approximate group $C=\pi(A)\subset G/H$ and a $2K^3$-approximate group $B_1=A^2\cap H\subset H$ such that $A\subset\phi(C)B_1$.
\end{remark}
\begin{proof}[Proof of Lemma \ref{lem:splitting}]
The proof is all essentially contained in \cite[Lemma 7.7]{tao.product.set}. It is trivial that there exists a right inverse $\phi$ to $\pi$ satisfying conditions (i) and (ii); we claim that $\phi$ must additionally satisfy conditions (iii) and (iv).

To see that $\phi$ satisfies condition (iii), note that if $x_1\cdots x_n=1$ in $G/H$ then this implies by definition that $\phi(x_1^{\varepsilon_1})^{\varepsilon_1}\cdots\phi(x_n^{\varepsilon_n})^{\varepsilon_n}$ lies in $H$. The fact that it also lies in $A^{rn}$ is immediate from the symmetry of $A$ and condition (i).

To verify condition (iv), note that for every $a\in A$ it follows from condition (i) and the symmetry of $A$ that $\phi(\pi(a))^{-1}a\in A^2$, and hence that $a\in \phi(\pi(a))A^2$. Furthermore, the fact that $a$ and $\phi(\pi(a))$ lie in the same fibre of $\pi$ implies that $a\in\phi(\pi(a))H$. Thus $A\subset\phi(\pi(A))(A^2\cap H)$, as required.
\end{proof}
\begin{remark}If $G$ has no 2-torsion then we may insist, as was done in \cite{tao.product.set}, that $\phi(x^{-1})=\phi(x)^{-1}$ for every $x$, in which case the exponents $\varepsilon_i$ featuring in condition (iii) will be superfluous.
\end{remark}
When applying the splitting lemma we shall repeatedly use the following observation; it is essentially trivial, but we record it as a lemma for ease of later reference.
\begin{lemma}\label{lem:size.phi(C)B}
Let $G$ be a group with a normal subgroup $H$. Write $\pi:G\to G/H$ for the canonical projection and suppose that $\phi:G/H\to G$ is a right inverse to $\pi$. Let $C\subset G/H$ and let $B\subset H$. Then $|\phi(C)B|=|C||B|$.
\end{lemma}
\begin{proof}The cosets $\phi(c)H$ for $c\in C$ are disjoint, and so in particular the sets $\phi(c)B$ are disjoint.
\end{proof}
\begin{corollary}\label{cor:size.of.B}Let $G$ be a group with a normal subgroup $H$ and write $\pi:G\to G/H$ for the canonical homomorphism. Suppose that $A\subset G$ is a finite symmetric set. Then $|\pi(A)||A^2\cap H|\ge|A|$.
\end{corollary}
\begin{proof}
This is immediate from Lemmas~\ref{lem:splitting} and \ref{lem:size.phi(C)B}.
\end{proof}
%
%
We close this section with the following converse to the splitting lemma.
\begin{lemma}[Converse to the splitting lemma]\label{lem:splitting.converse}
Let $G$ be a group with normal subgroup $H$, and write $\pi:G\to G/H$ for the canonical projection. Suppose that $C\subset G/H$ and $B_1\subset B_2\subset B_3\subset B_4\subset H$ are symmetric sets containing 1 such that
\begin{equation}\label{eq:splitting.converse.1}
|C^3|\le K_1|C|;
\end{equation}
\begin{equation}\label{eq:splitting.converse.2}
|B_4|\le K_2|B_1|;
\end{equation}
\begin{equation}\label{eq:splitting.converse.3}
|B_4^4|\le K_3|B_4|.
\end{equation}
Suppose further that $\phi:C^3\to G$ is a right inverse to $\pi$ such that
\begin{equation}\label{eq:splitting.converse.4}
\phi(x)B_i\phi(x)^{-1}\subset B_{i+1};\quad\phi(x)^{-1}B_i\phi(x)\subset B_{i+1}\quad\text{for all $x\in C$}
\end{equation}
and such that whenever $\varepsilon_i=\pm1$ we have
\begin{equation}\label{eq:splitting.converse.5}
\phi(x_1^{\varepsilon_1})^{\varepsilon_1}\phi(x_2^{\varepsilon_2})^{\varepsilon_2}\phi(x_3^{\varepsilon_3})^{\varepsilon_3}\phi(x_4^{\varepsilon_4})^{\varepsilon_4}\in B_4\qquad\text{for all $x_i\in C^3$ with $x_1x_2x_3x_4=1$.}
\end{equation}
Then we have
\[
|(\phi(C)B_1\cup B_1\phi(C)^{-1})^3|\le K_1K_2K_3|(\phi(C)B_1\cup B_1\phi(C)^{-1})|.
\]
In particular, by Lemma~\ref{lem:sm.trip.to.approx.grp}, $(\phi(C)B_1\cup B_1\phi(C)^{-1})^3$ is an $O((K_1K_2K_3)^{O(1)})$-approximate group.
\end{lemma}
\begin{remark}
Let us sketch why this can be thought of as a converse to Lemma \ref{lem:splitting}, at least in the case in which the set $A$ appearing in Lemma \ref{lem:splitting} is an approximate group. Set $C=\pi(A)$ and $B_1=A^2\cap H$ as in Remark \ref{rem:after.splitting}, and set $B_2=A^4\cap H$; $B_3=A^6\cap H$; and $B_4=A^{12}\cap H$. Let $\phi$ be the function given by Lemma \ref{lem:splitting}.

Lemma \ref{lem:splitting.converse} says that if conditions (\ref{eq:splitting.converse.1}), (\ref{eq:splitting.converse.2}), (\ref{eq:splitting.converse.3}), (\ref{eq:splitting.converse.4}) and (\ref{eq:splitting.converse.5}) hold then $\phi(C)(B_1)$ is close to being an approximate group. On the other hand, if $A$ is an approximate group, which by condition (iv) of Lemma \ref{lem:splitting} roughly means that $\phi(C)(B_1)$ is close to being an approximate group, then we leave it to the reader as a simple exercise to verify that conditions (\ref{eq:splitting.converse.1}), (\ref{eq:splitting.converse.2}), (\ref{eq:splitting.converse.3}), (\ref{eq:splitting.converse.4}) and (\ref{eq:splitting.converse.5}) hold.
\end{remark}
\begin{proof}[Proof of Lemma \ref{lem:splitting.converse}]
This, again, is essentially found in \cite[Lemma 7.7]{tao.product.set}. Let $y_1,y_2,y_3\in(\phi(C)B_1\cup B_1\phi(C)^{-1})$ and write $x_i:=\pi(y_i)$; then either $y_i\in\phi(x_i)B_1$ or $y_i\in B_1\phi(x_i^{-1})^{-1}$. In each case, repeated application of (\ref{eq:splitting.converse.4}) implies that there exist $\varepsilon_i=\pm1$ such that
\[
y_1y_2y_3\in\phi(x_1^{\varepsilon_1})^{\varepsilon_1}\phi(x_2^{\varepsilon_2})^{\varepsilon_2}\phi(x_3^{\varepsilon_3})^{\varepsilon_3}B_4^3.
\]
An application of (\ref{eq:splitting.converse.5}) then implies that
\begin{equation}\label{eq:splitting.converse.6}
y_1y_2y_3\in\phi(x_1x_2x_3)B_4^4\subset\phi(C^3)B_4^4.
\end{equation}
Lemma~\ref{lem:size.phi(C)B} combines with (\ref{eq:splitting.converse.1}), (\ref{eq:splitting.converse.2}) and (\ref{eq:splitting.converse.3}) to imply that $|\phi(C^3)B_4^4|\le K_1K_2K_3|\phi(C)B_1|$, and in conjunction with (\ref{eq:splitting.converse.6}) this gives the result immediately.
\end{proof}
\section{Commutators}\label{sec:comm.defs}
In this section we introduce \emph{commutators}, which are central to the theory of nilpotent groups. Following a set up in \cite[\S1]{tor.free.nilp}, which was in turn based on \cite[\S11.1]{hall}, we introduce the following definition.
\begin{definition}[Commutators and weights]
We define \emph{(formal) commutators} in the letters $x_1,\ldots,x_r$ recursively by defining each $x_i$ to be a formal commutator and, given two formal commutators $\alpha,\alpha'$ in the $x_j$, defining $[\alpha,\alpha']$ also to be a formal commutator.

To each commutator $\alpha$ we assign a \emph{weight vector} $\chi(\alpha)\in\N_0^r$, defined recursively by setting $\chi(x_i):=e_i$ and, given two formal commutators $\alpha,\alpha'$ in the $x_j$, defining $\chi([\alpha,\alpha'])=\chi(\alpha)+\chi(\alpha')$. We define the \emph{total weight} $|\chi(\alpha)|$ of a commutator $\alpha$ to be $\|\chi(\alpha)\|_1$. Given a weight vector $\chi\in\N_0^r$ and a vector $L=(L_1,\ldots,L_r)$ of positive integers we write $L^\chi$ to denote the quantity $L_1^{\chi_1}\cdots L_r^{\chi_r}$.

Noting that this results in at most finitely many commutators of any given weight vector, we assign a fixed total ordering $\prec$ to the commutators, chosen arbitrarily subject to the conditions that $x_1\prec\ldots\prec x_r$, that commutators of the same weight vector are consecutive, and that commutators of lower total weight come before commutators of higher total weight.

Finally, for each commutator $\alpha$ we define the \emph{(formal) inverse commutator} $\alpha^{-1}$. 
We extend $\prec$ to a partial ordering of commutators and their formal inverses, defining $\alpha^{\pm1}\prec\beta^{\pm1}$ when $\alpha\prec\beta$.
\end{definition}
Of course, for elements $u,v$ in a group the commutator $[u,v]$ is interpreted as the group element $u^{-1}v^{-1}uv$, and the formal inverse of a commutator $\alpha$ is interpreted as the inverse of the interpretation of $\alpha$. Thus if $x_1,\ldots,x_r$ are group elements then the formal commutators in the $x_i$, and their inverses, have interpretations as group elements.\footnote{Of course, when a commutator $[\alpha,\beta]$ is interpreted as a group element then the inverse of that group element is equal to the interpretation of $[\beta,\alpha]$. As formal commutators, however, $[\beta,\alpha]$ and $[\alpha,\beta]^{-1}$ are considered distinct.} In the case where the $x_i$ are elements of a \emph{nilpotent} group there are only finitely many commutators with non-trivial interpretations as group elements.

It will be useful to distinguish between, for example, the commutator $[x_1,x_2]$ and the function $(x_1,x_2)\mapsto[x_1,x_2]$, which we shall call a \emph{commutator form}. We therefore also introduce the following definition.
\begin{definition}[Commutator form]\label{def:com.form}
A \emph{commutator form} of weight $n$ on the letters $x_1,\ldots,x_r$ will be a function from $\{x_1,\ldots,x_r\}^n$ to the set of formal commutators in the $x_i$, defined recursively as follows. The identity function $\iota$ given by $\iota(x_i):=x_i$ is the unique commutator form of weight 1. The commutator forms of weight $n$ will consist of all functions $\gamma$ for which there exist commutator forms $\gamma_1,\gamma_2$ with respective weights $m_1,m_2$ summing to $n$, and a permutation $\sigma\in S_n$, such that
\[
\gamma(x_{i_1},\ldots,x_{i_n})=[\gamma_1(x_{i_{\sigma(1)}},\ldots,x_{i_{\sigma(m_1)}}),\gamma_2(x_{i_{\sigma(m_1+1)}},\ldots,x_{i_{\sigma(n)}})]
\]
for all $i_j\in[1,r]$. We will denote by $|\chi(\gamma)|$ the weight of a commutator form $\gamma$.
\end{definition}
\begin{remark}
The permutation $\sigma$ is necessary in the definition of a commutator form only to ensure, for example, that the function $\gamma:(x_1,x_2,x_3)\mapsto[[x_1,x_3],x_2]$ is not excluded for the rather artificial reason that the elements $x_1,x_2,x_3$ appear out of order in the definition of the function.
\end{remark}
\begin{definition}[Simple commutators]\label{def:simp.coms}
Given elements $x_1,x_2,\ldots$ of a group $G$ we define the \emph{simple commutators} $[x_1,\ldots,x_k]$ inductively by
\[
[x_1,x_2]:=x_1^{-1}x_2^{-2}x_1x_2;\qquad[x_1,\ldots,x_k]:=[[x_1,\ldots,x_{k-1}],x_k]\quad(k\ge2).
\]
If $X_1,\ldots,X_s$ are subsets of $G$ then we set
\[
[X_1,\ldots,X_s]:=\{[x_1,\ldots,x_s]:x_i\in X_i\}.
\]
If the total weight of a simple commutator is not clear from context then we indicate it by a subscript to the commutator bracket; thus, for example, if $X$ is a subset of $G$ then we write
\[
[X,\ldots,X]_s:=\{[x_1,\ldots,x_s]:x_i\in X\}.
\]
\end{definition}
\begin{remark}
The notation $[X_1,\ldots,X_s]$ defined in Definition \ref{def:simp.coms} is unfortunately inconsistent with the convention often used in the literature whereby the commutator $[H_1,\ldots,H_s]$ of \emph{subgroups} $H_1,\ldots,H_s$ of $G$ is defined to be the subgroup \emph{generated by} $\{[h_1,\ldots,h_s]:h_i\in H_i\}$. We trust that this will not be too confusing for the reader, and note in particular that whenever we use this notation the sets $X_i$ will not, in general, be subgroups, and so the aforementioned convention would in any case not apply.
\end{remark}
\begin{definition}[Lower central series]
Given a group $G$ we define the \emph{lower central series} $\Gamma_1(G),\Gamma_2(G),\ldots$ of $G$ by
\[
\Gamma_1(G):=G;\qquad\Gamma_{i+1}(G):=\langle[a,b]\text{ }|\text{ }a\in\Gamma_i(G),b\in G\rangle.
\]
\end{definition}
A group $G$ is of course nilpotent of step at most $s$ if and only if $\Gamma_{s+1}(G)=\{1\}$, and if $\alpha$ is a commutator of total weight $n$ in the elements of $G$ then $\alpha\in\Gamma_n(G)$.
We note the following characterisation of subsets of $G$ that generate subgroups of step $\tilde s$, which follows, for example, from \cite[Theorem 10.2.3]{hall}.
\begin{lemma}\label{lem:nilp.gens}
Let $G$ be a nilpotent group, and suppose that $X$ is a finite subset of $G$. Then $X$ generates a subgroup of step at most $\tilde s$ if and only if $[X,\ldots,X]_{\tilde s+1}=\{1\}$.
\end{lemma}
\section{The torsion-free case}\label{sec:ind.tor.free}
Our aim is to use the splitting lemma (Lemma \ref{lem:splitting}) to express $A$ in terms of approximate groups of lower nilpotency class. To motivate matters, in this section we consider the torsion-free case, where the details are cleanest. The key result is the following.
\begin{prop}\label{prop:tor.free.inside.prod.of.lower.step}
Let $G$ be a torsion-free $s$-step nilpotent group with $s\ge2$ and suppose that $A\subset G$ is a $K$-approximate subgroup. Then there exist an integer $r\le K^{O(1)}$ and $K^{O(1)}$-approximate groups $A_1,\ldots,A_r\subset A^{O(1)}$, each of which generates a subgroup of step less than $s$, such that
\[
A\subset A_1\cdots A_r.
\]
\end{prop}
We deduce Proposition~\ref{prop:tor.free.inside.prod.of.lower.step} from the following result.
\begin{prop}\label{prop:prod.of.approx.grps.tor.free}
Let $G$ be a torsion-free $s$-step nilpotent group with $s\ge2$ and suppose that $A\subset G$ is a $K$-approximate subgroup. Then there exist an integer $r\le K^{O(1)}$ and $K^{O(1)}$-approximate groups $X_0,\ldots,X_r\subset A^{O(1)}$ of step less than $s$, such that
\[
|X_0\cdots X_r|\ge\exp(-K^{O(1)})|A|.
\]
\end{prop}
A significant portion of the proof of Proposition \ref{prop:prod.of.approx.grps.tor.free} is in fact valid without the assumption that $G$ is torsion free. We isolate this portion as follows.
\begin{prop}\label{prop:syros}
Let $G$ be an $s$-step nilpotent group with $s\ge2$ and write $\pi:G\to G/\Gamma_2(G)$ for the canonical projection. Suppose that $A\subset G$ is a $K$-approximate subgroup. Then there exist an integer $r\le K^{O(1)}$ and $K^{O(1)}$-approximate groups $X_1,\ldots,X_r\subset A^{O(1)}$ of step less than $s$, and a $K^{O(1)}$-approximate group $X_0\subset A^{O(1)}$ with the property that $\pi(X_0)$ is a finite subgroup $\tilde H$ inside $\pi(A^4)$, such that
\[
|X_0\cdots X_r|\ge\exp(-K^{O(1)})|A|.
\]
\end{prop}
\begin{proof}
The set $\pi(A)$ is trivially a $K$-approximate group, and so by Theorem~\ref{thm:green-ruzsa} and the symmetry of $A$ there is an abelian progression $P=P(x_1,\ldots,x_r;L)$ of dimension
\begin{equation}\label{eq:key.1}
r\le K^{O(1)}
\end{equation}
and a subgroup $\tilde H$ of $G/\Gamma_2(G)$ such that
\begin{equation}\label{eq:key.1.5}
\tilde H+P\subset\pi(A^4)
\end{equation}
and
\begin{equation}\label{eq:key.2}
|\tilde H+P|\ge\exp(-K^{O(1)})|\pi(A)|.
\end{equation}
Set
\[
Y_n:=A^n\cap\Gamma_2(G),
\]
and for each $i=1,\ldots,r$ write
\[
P_i:=\{lx_i:|l|\le L_i\}
\]
so that $P=P_1+\ldots+P_r$. Let $\phi$ be a right inverse $\phi:G/\Gamma_2(G)\to G$ to $\pi$ given by applying Lemma~\ref{lem:splitting}. By conditions (i) and (iii) of that lemma and (\ref{eq:key.1.5}), for any $l_1,\ldots,l_r$ satisfying $|l_i|\le L_i$ and $h\in\tilde H$ we have
\begin{align*}
\phi(h+l_1x_1+\ldots+l_rx_r) &\in\phi(h+l_1x_1+\ldots+l_{r-1}x_{r-1})\phi(l_rx_r)Y_{12}\\
                             &\subset\phi(h+l_1x_1+\ldots+l_{r-2}x_{r-2})Y_{12}\phi(l_{r-1}x_{r-1})Y_{12}\phi(l_rx_r)Y_{12}\\
                             &\qquad\vdots\\
                             &\subset\phi(h)\phi(l_1x_1)Y_{12}\cdots\phi(l_{r-1}x_{r-1})Y_{12}\phi(l_rx_r)Y_{12},
\end{align*}
and hence
\begin{equation}\label{eq:key.2.2}
\phi(\tilde H+P)\subset\phi(\tilde H)\phi(P_1)Y_{12}\cdots\phi(P_{r-1})Y_{12}\phi(P_r)Y_{12}
\end{equation}
Now (\ref{eq:key.1.5}), condition (i) of Lemma~\ref{lem:splitting} and the symmetry of $A$ imply that
\begin{equation}\label{eq:key.2.3}
\phi(y)^{-1}Y_n\phi(y)\subset Y_{n+8};\quad\phi(y)Y_n\phi(y)^{-1}\subset Y_{n+8}\quad\text{for all $y\in P$},
\end{equation}
and so from (\ref{eq:key.2.2}) we can conclude that
\[
\phi(\tilde H+P)\subset\phi(\tilde H)Y_{20}\phi(P_1)Y_{20}\cdots\phi(P_{r-1})Y_{20}\phi(P_r),
\]
and hence that
\begin{equation}\label{eq:key.3}
\phi(\tilde H+P)Y_{20}\subset\phi(\tilde H)Y_{20}\phi(P_1)Y_{20}\cdots\phi(P_r)Y_{20}.
\end{equation}
Now Lemma~\ref{lem:approx.grp.in.[g,g]} implies that $|Y_n^4|\le K^{O(n)}|Y_n|$ and that $|Y_n|\le K^n|Y_{n'}|$ for every $n'\ge2$, and of course we have $|3P_i|\le3|P_i|$ and $|3\tilde H|=|\tilde H|$. Combined with (\ref{eq:key.2.3}) and condition (iii) of Lemma~\ref{lem:splitting}, this means that we may apply Lemma~\ref{lem:splitting.converse} with $C=\tilde H$ or $P_i$, $B_1=Y_{20}$, $B_2=Y_{28}$, $B_3=Y_{36}$ and $B_4=Y_{48}$ to conclude that
\[
X_0:=(\phi(\tilde H)Y_{20}\cup Y_{20}\phi(\tilde H)^{-1})^3
\]
and
\[
X_i:=(\phi(P_i)Y_{20}\cup Y_{20}\phi(P_i)^{-1})^3
\]
are $K^{O(1)}$-approximate groups.

By the definitions of $\phi$ and $Y_n$ we have $X_i\subset A^{O(1)}$, as desired. To see that $|X_0\cdots X_r|\ge\exp(-K^{O(1)})|A|$, note first that by (\ref{eq:key.3}) we have $|X_0\cdots X_r|\ge|\phi(\tilde H+P)Y_{20}|$, which is equal to $|\tilde H+P||Y_{20}|$ by Lemma~\ref{lem:size.phi(C)B}. This in turn is at least $\exp(-K^{O(1)})|\pi(A)||Y_{20}|$ by (\ref{eq:key.2}), which is, as desired, at least $\exp(-K^{O(1)})|A|$ by Corollary~\ref{cor:size.of.B}.

Since $\phi$ is a right inverse to $\pi$, the image $\pi(X_0)$ is equal to the subgroup $\tilde H$ of $G/\Gamma_2(G)$, and by (\ref{eq:key.1.5}) it is finite and contained in $\pi(A^4)$, as desired. The fact that $\phi$ is a right inverse to $\pi$ also implies that for $|l|\le L_i$ we have $\phi(lx_i)\in\phi(x_i)^l\Gamma_2(G)$. In particular, $\phi(P_i)Y_{20}$, and hence $X_i$, is a subset of the group $\langle\phi(x_i)\rangle\Gamma_2(G)$, which is of step at most $s-1$. This completes the proof.
\end{proof}
Given Proposition \ref{prop:syros}, in order to prove Proposition \ref{prop:prod.of.approx.grps.tor.free} it will suffice to show that if $X_0$ is a subset of a torsion-free $s$-step nilpotent group, and if the image $\pi(X_0)$ of $X_0$ is a finite subgroup of $G/\Gamma_2(G)$, then $X_0$ generates a subgroup of step at most $s-1$. Equivalently, thanks to Lemma \ref{lem:nilp.gens}, it will suffice to show that the set $[X_0,\ldots,X_0]_s$ contains only the identity.

The key to showing this turns out to be the following statement, which is an immediate corollary of Proposition \ref{prop:com.components}. 
\begin{lemma}\label{lem:comms.are.homs}
If $G$ is an $s$-step nilpotent group then the map $\alpha:G^s\to G$ defined by $\alpha(x_1,\ldots,x_s)=[x_1,\ldots,x_s]$ is a homomorphism in each variable. Furthermore, $[x_1,\ldots,x_s]$ is trivial whenever there is some $i$ with $x_i\in\Gamma_2(G)$.
\end{lemma}
The final conclusion of Lemma \ref{lem:comms.are.homs} is particularly useful in conjunction with the following result, the proof of which is left as a simple exercise.
\begin{lemma}\label{lem:induced.multi.hom}
Let $k>0$ be an integer and let $\Gamma,\Gamma_1,\ldots,\Gamma_k$ be groups. For each $i=1,\ldots,k$ let $\Gamma_i'$ be a normal subgroup of $\Gamma_i$, and write $\rho_i:\Gamma_i\to\Gamma_i/\Gamma'_i$ for the canonical homomorphism. Suppose that $\varphi:\Gamma_1\times\cdots\times\Gamma_k\to\Gamma$ is a map that is a homomorphism in each variable, and suppose that $\varphi$ is trivial whenever there is an $i$ such that the $i$th argument of $\varphi$ belongs to $\Gamma_i'$. Then the function
\[
\begin{array}{ccccl}
\psi&:&\Gamma_1/\Gamma_1'\times\cdots\times\Gamma_k/\Gamma_k' &\to     &\Gamma\vspace{5pt}\\
    & &(\rho_1(x_1),\ldots,\rho_k(x_k))                       &\mapsto &\varphi(x_1,\ldots,x_k)
\end{array}
\]
is well defined and a homomorphism in each variable.

Moreover, if for each $i$ there is a subgroup $H_i$ of $\Gamma_i/\Gamma_i'$ and a set $V_i\subset\Gamma_i$ satisfying $\rho_i(V_i)=H_i$ then we have
\[
\varphi(V_1,\ldots,V_k)=\psi(H_1,\ldots,H_k).
\]
\end{lemma}
\begin{proof}[Proof of Proposition \ref{prop:prod.of.approx.grps.tor.free}]
Let $X_0,\ldots,X_r$ be the subsets of $G$, and $\tilde H$ the finite subgroup of $G/\Gamma_2(G)$, given by Proposition \ref{prop:syros}. By Proposition \ref{prop:syros} it suffices to show that $[X_0,\ldots,X_0]_s=\{1\}$. However, it follows from Lemmas \ref{lem:comms.are.homs} and \ref{lem:induced.multi.hom} that there is some function
\[
\psi:G/\Gamma_2(G)\times\cdots\times G/\Gamma_2(G)\to G
\]
that is a homomorphism in each variable and such that
\begin{equation}\label{eq:syros}
[X_0,\ldots,X_0]_s=\psi(\tilde H,\ldots,\tilde H)=\bigcup_{h_i\in\tilde H}\psi(\tilde H,h_2,\ldots,h_s).
\end{equation}
Since $\psi$ is a homomorphism in the first variable, each set $\psi(\tilde H,h_2,\ldots,h_s)$ is a finite subgroup, and the fact that $G$ is torsion free implies that each of these finite subgroups is trivial. Given (\ref{eq:syros}), this implies that $[X_0,\ldots,X_0]_s=\{1\}$, as desired, and so the proposition is proved.
\end{proof}
\begin{proof}[Proof of Proposition~\ref{prop:tor.free.inside.prod.of.lower.step}]
Proposition~\ref{prop:prod.of.approx.grps.tor.free} gives an integer
\begin{equation}\label{eq:tor.free.inside.prod.of.lower.step.1}
k\le K^{O(1)}
\end{equation}
and $K^{O(1)}$-approximate groups
\begin{equation}\label{eq:tor.free.inside.prod.of.lower.step.2}
X_0,\ldots,X_k\subset A^{O(1)},
\end{equation}
each of which generates a subgroup of step less than $s$, such that
\begin{equation}\label{eq:tor.free.inside.prod.of.lower.step.3}
|X_0\cdots X_k|\ge\exp(-K^{O(1)})|A|.
\end{equation}
By (\ref{eq:tor.free.inside.prod.of.lower.step.1}) and (\ref{eq:tor.free.inside.prod.of.lower.step.2}) we also have
\begin{equation}\label{eq:tor.free.inside.prod.of.lower.step.4}
X_0\cdots X_k\subset A^{K^{O(1)}}.
\end{equation}
In light of (\ref{eq:tor.free.inside.prod.of.lower.step.3}) and (\ref{eq:tor.free.inside.prod.of.lower.step.4}), applying Proposition~\ref{prop:chang} to $X_0\cdots X_k$ with $M=K^{O(1)}$ and $M'=\exp(K^{O(1)})$ then yields sets
\begin{equation}\label{eq:tor.free.inside.prod.of.lower.step.4.5}
S_1,\ldots,S_t\subset A
\end{equation}
with
\begin{equation}\label{eq:tor.free.inside.prod.of.lower.step.5}
t\le K^{O(1)}
\end{equation}
and
\begin{equation}\label{eq:tor.free.inside.prod.of.lower.step.6}
|S_i|\le2K
\end{equation}
such that
\begin{equation}\label{eq:tor.free.inside.prod.of.lower.step.7}
A\subset S_{t-1}^{-1}\cdots S_1^{-1}X_k\cdots X_0X_0\cdots X_kS_1\cdots S_t.
\end{equation}
Let $s_1,s_2,\ldots,s_m$ be a list of all elements of $S_1$, followed by those of $S_2$, then those of $S_3$ and so on all the way up to the elements of $S_t$, noting that if an element belongs to more than one set $S_i$ then it will appear on the list more than once. Note that (\ref{eq:tor.free.inside.prod.of.lower.step.5}) and (\ref{eq:tor.free.inside.prod.of.lower.step.6}) imply that
\begin{equation}\label{eq:tor.free.inside.prod.of.lower.step.8}
m\le K^{O(1)}.
\end{equation}
It follows from (\ref{eq:tor.free.inside.prod.of.lower.step.7}) that
\begin{equation}\label{eq:tor.free.inside.prod.of.lower.step.9}
A\subset P(s_m;1)\cdots P(s_1;1)X_k\cdots X_0X_0\cdots X_kP(s_1;1)\cdots P(s_m;1).
\end{equation}
For each $i$ the set $P(s_i;1)$ is an abelian 2-approximate group, and by (\ref{eq:tor.free.inside.prod.of.lower.step.4.5}) it lies inside $A$, and so the desired result follows from (\ref{eq:tor.free.inside.prod.of.lower.step.1}), (\ref{eq:tor.free.inside.prod.of.lower.step.2}), (\ref{eq:tor.free.inside.prod.of.lower.step.8}) and (\ref{eq:tor.free.inside.prod.of.lower.step.9}).
\end{proof}
\begin{corollary}\label{cor:tor.free.inside.prod.of.ab}
Let $G$ be a torsion-free $s$-step nilpotent group and suppose that $A\subset G$ is a $K$-approximate subgroup. Then there exist an integer $r\le K^{O_s(1)}$ and abelian $K^{O_s(1)}$-approximate groups $A_1,\ldots,A_r\subset A^{O_s(1)}$ such that
\[
A\subset A_1\cdots A_r.
\]
\end{corollary}
\begin{proof}
This follows from Proposition~\ref{prop:tor.free.inside.prod.of.lower.step} and induction on $s$.
\end{proof}
By Theorem~\ref{thm:ab.freiman}, of course, we can place each abelian approximate group $A_i$ given by Corollary~\ref{cor:tor.free.inside.prod.of.ab} inside an abelian progression, and hence place $A$ itself inside a product of abelian progressions. We now define a structure that encompasses such a product.
\begin{definition}[Ordered progression]\label{def:ord.prog}Let $x_1,\ldots,x_r\in G$ and let $L_1,\ldots,L_r\in\N$. Then the \emph{ordered progression} on generators $x_1,\ldots,x_r$ with lengths $L_1,\ldots,L_r$ is defined to be
\[
P_\textup{ord}(x_1,\ldots,x_r;L):=\{x_1^{l_1}\cdots x_r^{l_r}:|l_i|\le L_i\}.
\]
Define $r$ to be the \emph{rank} of $P_\textup{ord}(x_1,\ldots,x_r;L)$.
\end{definition}
The nilprogressions, nilpotent progressions and ordered progressions defined in Definitions~\ref{def:nilprog}, \ref{def:nilp.prog} and \ref{def:ord.prog} are not, in general, the same objects. Proposition \ref{prop:prim.prog.to.nilp.prog} in Appendix \ref{sec:rough.equiv.prog}, however, shows that if $G$ is an $s$-step nilpotent group, and if $x_1,\ldots,x_r\in G$ and $L_1,\ldots,L_r\in\N$, then
\begin{equation}\label{eq:equiv.of.progs}
P_\textup{ord}(x_1,\ldots,x_r;L)\subset P^*(x_1,\ldots,x_r;L)\subset P(x_1,\ldots,x_r;L)\subset P_\textup{ord}(x_1,\ldots,x_r;L)^{r^{O_s(1)}},
\end{equation}
and so we can at least pass from one type of progression to another fairly efficiently.
\begin{proof}[Proof of Theorem~\ref{thm:freiman.tor.free.nilp}]By Corollary~\ref{cor:tor.free.inside.prod.of.ab} there exist an integer
\begin{equation}\label{eq:freiman.tor.free.nilp.0}
r\le K^{O_s(1)}
\end{equation}
and abelian $K^{O_s(1)}$-approximate groups
\begin{equation}\label{eq:freiman.tor.free.nilp.1}
A_1,\ldots,A_r\subset A^{O_s(1)}
\end{equation}
satisfying
\begin{equation}\label{eq:freiman.tor.free.nilp.2}
A\subset A_1\cdots A_r.
\end{equation}
Theorem~\ref{thm:ab.freiman} and the assumption that $G$ is torsion free imply that there exist, for each $j=1,\ldots,r$, an abelian progression $P^{(j)}:=P(x_1^{(j)},\ldots,x_{k_j}^{(j)};L^{(j)})$ satisfying
\begin{equation}\label{eq:freiman.tor.free.nilp.3}
k_j\le K^{O_s(1)}
\end{equation}
and
\begin{equation}\label{eq:freiman.tor.free.nilp.4}
A_j\subset P^{(j)}\subset A_j^{K^{O_s(1)}}.
\end{equation}
Combining (\ref{eq:freiman.tor.free.nilp.0}), (\ref{eq:freiman.tor.free.nilp.1}), (\ref{eq:freiman.tor.free.nilp.2}) and (\ref{eq:freiman.tor.free.nilp.4}) then implies that $A\subset P^{(1)}\cdots P^{(r)}\subset A^{K^{O_s(1)}}$, and so writing
\[
(y_1,\ldots,y_m):=(x_1^{(1)},\ldots,x_{k_1}^{(1)},x_1^{(2)},\ldots,x_{k_2}^{(2)},\ldots,x_1^{(r)},\ldots,x_{k_r}^{(r)})
\]
and
\[
(N_1,\ldots,N_m):=(L_1^{(1)},\ldots,L_{k_1}^{(1)},L_1^{(2)},\ldots,L_{k_2}^{(2)},\ldots,L_1^{(r)},\ldots,L_{k_r}^{(r)}),
\]
gives
\begin{equation}\label{eq:freiman.tor.free.nilp.5}
A\subset P_\textup{ord}(y_1,\ldots,y_m;N)\subset A^{K^{O_s(1)}}.
\end{equation}
Furthermore,
\begin{equation}\label{eq:freiman.tor.free.nilp.6}
m\le K^{O_s(1)}
\end{equation}
by (\ref{eq:freiman.tor.free.nilp.0}) and (\ref{eq:freiman.tor.free.nilp.3}). Theorem~\ref{thm:freiman.tor.free.nilp} then follows from (\ref{eq:equiv.of.progs}), (\ref{eq:freiman.tor.free.nilp.5}) and (\ref{eq:freiman.tor.free.nilp.6}).
\end{proof}
\section{Strategy of the general argument}\label{sec:strategy}
Let us pause at this point to consider how the argument we used in proving Theorem \ref{thm:freiman.tor.free.nilp} might be adapted to work in an arbitrary nilpotent group. The key observation of the torsion-free argument was Proposition \ref{prop:tor.free.inside.prod.of.lower.step}, which essentially allowed us to express an approximate group of step $s$ as a product of $O_K(1)$ approximate groups of step at most $s-1$. However, such a statement is simply not true in the absence of the torsion-free assumption, as illustrated by the following example.
\begin{example}\label{eg:prod.of.lower.step}Let $F$ be the free product of $n$ copies of the cyclic group with two elements, and let $G$ be the quotient $F/\Gamma_3(F)$. Then $G$ is a genuine $2$-step nilpotent group and, as a finitely generated nilpotent group each of whose generators is of finite order, it is finite. In particular, $G$ is certainly a $2$-step $K$-approximate group for any $K\ge1$. However, $G$ cannot be expressed as a product of $O(1)$ abelian sets as $n\to\infty$.
\end{example}
Having observed that the statement of Proposition \ref{prop:tor.free.inside.prod.of.lower.step} fails in an arbitrary nilpotent group, let us examine where the proof breaks down. The key to the proof of Proposition \ref{prop:tor.free.inside.prod.of.lower.step} was Proposition \ref{prop:prod.of.approx.grps.tor.free}, which allowed us to place a product of $K^{O(1)}$ approximate groups densely inside $A^{K^{O(1)}}$. The first part of that argument was valid in an arbitrary nilpotent group, and separated out as Proposition \ref{prop:syros}. However, the final step in the proof of Proposition \ref{prop:prod.of.approx.grps.tor.free}, in which we showed, in the notation used there, that $[X_0,\ldots,X_0]_s=\{1\}$, relied on the assumption that $G$ was torsion free. 

Nonetheless, in the context of Theorem~\ref{thm:freiman.p-group} all is not lost, as whilst Theorem~\ref{thm:freiman.p-group} is stronger than Theorem \ref{thm:freiman.tor.free.nilp}, the conclusion of Theorem~\ref{thm:freiman.p-group} in a group with elements of finite order is, in a certain sense, weaker than the conclusion of Theorem \ref{thm:freiman.tor.free.nilp}. Indeed, where Theorem \ref{thm:freiman.tor.free.nilp} required us to show that a $K$-approximate group $A$ was dense inside a nilprogression $P^*$, Theorem~\ref{thm:freiman.p-group} requires us only to show that the image of $A$ in a quotient by some normal subgroup $H\subset A^{K^{O(1)}}$ is dense inside a nilprogression. This allows us to make the following reduction.
\begin{lemma}\label{lem:normal.reduction}Let $G$ and $A$ be as in Theorem~\ref{thm:freiman.p-group}, and assume additionally that $G$ is generated by $A$. Suppose that $N\lhd G$ is a normal subgroup and that $N\subset A^M$. Suppose Theorem~\ref{thm:freiman.p-group} holds in the quotient $G/N$ with some implied constants. Then Theorem~\ref{thm:freiman.p-group} holds for $A$ in $G$ with the implied constants increased by at most $\log_2M$.
\end{lemma}
\begin{proof}[Proof sketch]Write $\rho:G\to G/N$ for the canonical projection. Since $\rho(A)$ is a $K$-approximate group, the fact that Theorem~\ref{thm:freiman.p-group} holds in $G/N$ implies that there is a normal subgroup $H$ of $G/N$ and a nilprogression $P^*\subset G/N$ such that
\[
\rho(A)\subset HP^*\subset\rho(A^{K^{O(1)}}).
\]
However, the fact that $N\subset A^M$ implies that
\[
A\subset\rho^{-1}(HP^*)\subset A^{K^{O(1)}+M},
\]
and it is straightforward to check that $\hat H=\rho^{-1}(H)$ is a normal subgroup of $G$ and that there is a nilprogression $\hat{P^*}\subset G$ of the same rank as $P^*$ and with the property that $\rho^{-1}(HP^*)=\hat H\hat{P^*}$.
\end{proof}
The upshot of this is that if there is some subset $B$ of $A$ that is difficult to analyse, then it is sufficient to find some normal subgroup $N$ of $G$ inside $A^{K^{O(1)}}$ that contains $B$ and restrict attention to the quotient $G/N$, in which the image of $B$ is trivial.

It turns out that this is indeed possible in the case of the troublesome set $X_0$, and ultimately yields the following generalisation of Proposition \ref{prop:tor.free.inside.prod.of.lower.step}.
\begin{prop}[Key proposition]\label{prop:p.grp.inside.prod.of.lower.step}
Let $m>0,s,\hat s\ge2$ be integers. Let $G$ be an $s$-step nilpotent group generated by a $K$-approximate group $A$, and let $\tilde A$ be an $\tilde s$-step $\tilde K$-approximate group inside $A^m$. Then there exist an integer $r\le\tilde K^{O(1)}$, a normal subgroup $N$ inside $A^{K^{O_{m,s}(1)}}$ and $\tilde K^{O(1)}$-approximate groups $A_1,\ldots,A_r\subset\tilde A^{O(1)}$ such that $[A_i,\ldots,A_i]_{\tilde{s}}\subset N$ for all $i$, and such that
\[
\tilde A\subset A_1\cdots A_r.
\]
\end{prop}
\begin{remark}If $G$ is assumed to be torsion free in Proposition \ref{prop:p.grp.inside.prod.of.lower.step} then the finite subgroup $N$ is automatically trivial, and so we recover the statement of Proposition~\ref{prop:tor.free.inside.prod.of.lower.step}.
\end{remark}
The next few sections are largely dedicated to the proof of Proposition \ref{prop:p.grp.inside.prod.of.lower.step}, with the conclusion of that proof appearing in Section \ref{sec:p.group.proof}.

Before we embark on the details in full, let us give a flavour of the overall argument in a particularly simple case, namely that of a 2-step $p$-group. In that case, our aim is essentially to identify a normal subgroup $N\lhd G$ inside $A^{K^{O(1)}}$ that contains $[X_0,X_0]$. Since $[X_0,X_0]$ is central, the group $\langle[X_0,X_0]\rangle$ generated by $[X_0,X_0]$ will automatically be normal, and so in fact it is sufficient to show the following.
\begin{lemma}\label{lem:2.step.p.grp.quotient}
In the case that $G$ is a 2-step $p$-group generated by $A$ we have $\langle[X_0,X_0]\rangle\subset A^{K^{O(1)}}$.
\end{lemma}
Let us sketch a proof of Lemma \ref{lem:2.step.p.grp.quotient}. We know from the proof of Proposition \ref{prop:prod.of.approx.grps.tor.free} that $[X_0,X_0]$ is a union of finite subgroups of the abelian group $\Gamma_2(G)$. To deal with such sets we have the following result.
\begin{lemma}\label{lem:univ.subgroup}
Let $\Gamma$ be an abelian $p$-group of rank $r$ (written multiplicatively) and suppose that $X\subset\Gamma$ is a union of subgroups of $\Gamma$. Then
\[
\langle X\rangle\subset X^r.
\]
\end{lemma}
\begin{proof}
Lemma~\ref{lem:ab.subgroup.rank} implies that $\langle X\rangle$ is of rank at most $r$, and so Lemma~\ref{lem:burnside.basis} implies that there exist elements $x_1,\ldots,x_r\in X$ that generate $\langle X\rangle$. It follows that $\langle X\rangle=\langle x_1\rangle\cdots\langle x_r\rangle$. However, the assumption that $X$ is a union of subgroups implies that $\langle x_i\rangle\subset X$ for each $i$, and so $\langle x_1\rangle\cdots\langle x_r\rangle\subset X^r$ and the lemma is proved.
\end{proof}
\begin{remark}The assumption in Lemma \ref{lem:univ.subgroup} that $\Gamma$ is a $p$-group is necessary. The statement fails, for example, if $\Gamma=\Z/6\Z$ and $X$ is the union of the subgroups $\{0,2,4\}$ and $\{0,3\}$. This is not surprising in light of Remark \ref{rem:burnside.basis}.
\end{remark}
Now $[X_0,X_0]$ is a subset of $A^{O(1)}\cap\Gamma_2(G)$, which is a set with doubling constant at most $K^{O(1)}$ by Lemma \ref{lem:approx.grp.in.[g,g]}. Theorem \ref{thm:ab.freiman} therefore implies that there is a subgroup $U\subset A^{O(1)}\cap\Gamma_2(G)$ and a progression $Q$ of rank at most $K^{O(1)}$ such that $[X_0,X_0]\subset UQ$.

Now $U$ is central in $G$, and hence, in particular, normal in $G$, and so we may write $\tau:G\to G/U$ for the canonical projection. Note then that the rank bound on $Q$ implies that $[\tau(X_0),\tau(X_0)]$ lies in an abelian $p$-group of rank at most $K^{O(1)}$, and so it follows from Lemma \ref{lem:univ.subgroup} that
\[
\langle[\tau(X_0),\tau(X_0)]\rangle\subset[\tau(X_0),\tau(X_0)]^{K^{O(1)}},
\]
and hence that
\[
\langle[X_0,X_0]\rangle\subset[X_0,X_0]^{K^{O(1)}}U\subset A^{K^{O(1)}}.
\]
This was precisely the claim of Lemma \ref{lem:2.step.p.grp.quotient}, and so our sketch of the proof of that lemma is complete. In the next few sections we generalise these arguments to deal with arbitrary nilpotent groups.
\section{Images of multi-variable homomorphisms}\label{sec:img.of.multi.hom}
The last step of the proof of Proposition \ref{prop:prod.of.approx.grps.tor.free} required us to deal with the set $X_0$ resulting from Proposition \ref{prop:syros}. The key observation that allowed us to do so was Lemma \ref{lem:comms.are.homs}, which states that an arbitrary commutator form of weight $s$ is a homomorphism in each variable when applied to elements in an $s$-step nilpotent group. In this section we study such `multi-variable' homomorphisms in more detail. The main result is the following variant of Lemma \ref{lem:univ.subgroup}. Here, and throughout this section, abelian groups will be written multiplicatively for clarity.
\begin{prop}\label{prop:img.of.multi.hom}Let $k>0$ be an integer, let $\Gamma_1,\ldots,\Gamma_k$ be finite nilpotent groups and let $\Gamma$ be an abelian group of rank at most $r$. Let $\varphi:\Gamma_1\times\cdots\times\Gamma_k\to\Gamma$ be a map that is a homomorphism in each variable. Then
\[
\langle\varphi(\Gamma_1,\ldots,\Gamma_k)\rangle\subset\varphi(\Gamma_1,\ldots,\Gamma_k)^r.
\]
\end{prop}
We start with the following observation, which is essentially identical to an observation made in the previous section.
\begin{lemma}\label{lem:img.of.multi.hom}
Proposition \ref{prop:img.of.multi.hom} holds when $\Gamma$ is a $p$-group.
\end{lemma}
\begin{proof}This follows immediately from Lemma \ref{lem:univ.subgroup} and the observation that $\varphi(\Gamma_1,\ldots,\Gamma_k)$ can be expressed as a union of the subgroups $\varphi(\Gamma_1,x_2,\ldots,x_k)$ with $x_i\in\Gamma_i$.
\end{proof}
A key fact that will enable us to proceed from Lemma \ref{lem:img.of.multi.hom} to Proposition \ref{prop:img.of.multi.hom} is the following alternative characterisation of nilpotency in a finite group.
\begin{prop}[{\cite[Theorem 10.3.4]{hall}}]\label{prop:prod.of.p-grps}
A finite group is nilpotent if and only if it is a direct product of $p$-groups.
\end{prop}
The proof of Proposition \ref{prop:img.of.multi.hom} then rests on the following observation.
\begin{lemma}\label{lem:multi.hom.factors}
For $i=1,\ldots,k$ let $\Gamma_i=\prod_p\Gamma_i(p)$ be a direct product of $p$-groups $\Gamma_i(p)$, and let $\Gamma=\Gamma(0)\otimes\prod_p\Gamma(p)$ be a direct product of a torsion-free abelian group $\Gamma(0)$ and abelian $p$-groups $\Gamma(p)$. Let $\varphi:\Gamma_1\times\cdots\times\Gamma_k\to\Gamma$ be a map that is a homomorphism in each variable. Then whenever $u_i(p)\in\Gamma_i(p)$ we have
\begin{equation}\label{eq:img.of.p.in.p}
\varphi(u_1(p),\ldots,u_k(p))\in\Gamma(p)
\end{equation}
and
\begin{equation}\label{eq:img.of.multi.hom}
\varphi(\textstyle\prod_pu_1(p),\ldots,\textstyle\prod_pu_k(p))=\prod_p\varphi(u_1(p),\ldots,u_k(p)).
\end{equation}
\end{lemma}
\begin{proof}
\begin{sloppypar}Since $\varphi$ is a homomorphism in each variable, the order of $\varphi(u_1(p),\ldots,u_k(p))$ divides the order of each $u_i(p)$. One consequence of this is that the order of $\varphi(u_1(p),\ldots,u_k(p))$ is a power of $p$, and so (\ref{eq:img.of.p.in.p}) holds. Another consequence is that if $i<j$ and $p\ne q$ then
\[
\varphi(*,\ldots,*,u_i(p),*,\ldots,*,u_j(q),*,\ldots,*)=1,
\]
regardless of the starred entries, and so expanding $\varphi(\textstyle\prod_pu_1(p),\ldots,\textstyle\prod_pu_k(p))$ in each variable gives (\ref{eq:img.of.multi.hom}).
\end{sloppypar}
\end{proof}
\begin{proof}[Proof of Proposition \ref{prop:img.of.multi.hom}]
Proposition \ref{prop:prod.of.p-grps} implies that we can express each $\Gamma_i$ as direct products of $p$-groups, as in Lemma \ref{lem:multi.hom.factors}. Furthermore, the structure theorem for finitely generated abelian groups shows that we decompose $\Gamma$ as a direct product $\Gamma(0)\otimes\prod_p\Gamma(p)$ of a torsion-free abelian group $\Gamma(0)$ and abelian $p$-groups $\Gamma(p)$, also as in Lemma \ref{lem:multi.hom.factors}. Write $\pi_p$ for the canonical projection of $\Gamma$ onto $\Gamma(p)$, and note that if $X$ is a generating set for $\Gamma$ then $\pi_p(X)$ is a generating set for $\Gamma(p)$. In particular, this gives
\begin{equation}\label{eq:rk.gamma.p}
\rank\Gamma(p)\le r.
\end{equation}
Furthermore, $\pi_p\circ\varphi$ is a map from $\Gamma_1\times\cdots\times\Gamma_k$ to $\Gamma(p)$ that is a homomorphism in each variable, and so Lemma \ref{lem:img.of.multi.hom} combines with (\ref{eq:rk.gamma.p}) to give
\[
\langle\pi_p(\varphi(\Gamma_1,\ldots,\Gamma_k))\rangle\subset\pi_p(\varphi(\Gamma_1,\ldots,\Gamma_k))^r.
\]
Combined with Lemma \ref{lem:multi.hom.factors} this implies that
\begin{equation}\label{eq:true.in.factor}
\langle\varphi(\Gamma_1(p),\ldots,\Gamma_k(p))\rangle\subset\varphi(\Gamma_1(p),\ldots,\Gamma_k(p))^r.
\end{equation}
Repeated application of Lemma \ref{lem:multi.hom.factors}, combined with an application of (\ref{eq:true.in.factor}), therefore implies that
\begin{align*}
\langle\varphi(\Gamma_1,\ldots,\Gamma_k)\rangle&=\left\langle\varphi\left(\textstyle\prod_p\Gamma_1(p),\ldots,\textstyle\prod_p\Gamma_k(p)\right)\right\rangle\\
                                                &=\left\langle\textstyle\prod_p\varphi(\Gamma_1(p),\ldots,\Gamma_k(p))\right\rangle\\
                                                &=\textstyle\prod_p\langle\varphi(\Gamma_1(p),\ldots,\Gamma_k(p))\rangle\\
                                                &\subset\textstyle\prod_p(\varphi(\Gamma_1(p),\ldots,\Gamma_k(p))^r)\\
                                                &=\left(\textstyle\prod_p\varphi(\Gamma_1(p),\ldots,\Gamma_k(p))\right)^r\\
                                                &=\varphi\left(\textstyle\prod_p\Gamma_1(p),\ldots,\textstyle\prod_p\Gamma_k(p)\right)^r\\
                                                &=\varphi(\Gamma_1,\ldots,\Gamma_k)^r,
\end{align*}
as desired.
\end{proof}
\section{A decomposition into lower-step approximate groups}\label{sec:p.group.proof}
Just as in the proof of Proposition \ref{prop:tor.free.inside.prod.of.lower.step}, an important ingredient in the proof of Proposition~\ref{prop:p.grp.inside.prod.of.lower.step} is Proposition \ref{prop:syros}. However, in the more general setting of Proposition~\ref{prop:p.grp.inside.prod.of.lower.step}, in order to deal with the set $X_0$ arising from Proposition \ref{prop:syros} we need the following lemma.
\begin{prop}\label{prop:lower.step.in.quotient}
Let $G$ be an $s$-step nilpotent group generated by a $K$-approximate group $A$. Let $\tilde s\ge2$, let $\tilde G$ be an $\tilde s$-step nilpotent subgroup of $G$, and write $\tilde\pi:\tilde G\to\tilde G/\Gamma_2(\tilde G)$ for the canonical homomorphism. Suppose that $\tilde H$ is a finite subgroup of $\tilde G/\Gamma_2(\tilde G)$ and that $X_0\subset A^{\tilde m}\cap\tilde G$ is a set satisfying $\tilde\pi(X_0)=\tilde H$. Then there is a normal subgroup $N$ of $G$ such that
\[
[X_0,\ldots,X_0]_{\tilde s}\subset N\subset A^{K^{O_{\tilde m,s}(1)}}.
\]
\end{prop}
The main aim of this section is to prove Proposition \ref{prop:lower.step.in.quotient}. However, before we do so, let us see how it implies Proposition~\ref{prop:p.grp.inside.prod.of.lower.step}.
\begin{proof}[Proof of Proposition~\ref{prop:p.grp.inside.prod.of.lower.step}]
Combining Proposition \ref{prop:syros} and Proposition \ref{prop:lower.step.in.quotient} gives an integer $r\le\tilde K^{O(1)}$, a normal subgroup $N$ inside $A^{K^{O_{m,s}(1)}}$, and $\tilde K^{O(1)}$-approximate groups $X_0,\ldots,X_r\subset\tilde A^{O(1)}$ satisfying $[X_i,\ldots,X_i]_{\tilde s}\subset N$ for all $i$, and such that
\[
|X_0\cdots X_r|\ge\exp(-\tilde K^{O(1)})|\tilde A|.
\]
The result then follows from applying Chang's covering argument (Proposition \ref{prop:chang}) in exactly the same way as in the proof of Proposition \ref{prop:tor.free.inside.prod.of.lower.step}.
\end{proof}
Using the notation of Proposition \ref{prop:lower.step.in.quotient}, it follows from Lemmas \ref{lem:comms.are.homs} and \ref{lem:induced.multi.hom} that there is a function
\[
\alpha:\tilde G/\Gamma_2(\tilde G)\times\cdots\times\tilde G/\Gamma_2(\tilde G)\to G
\]
that is a homomorphism in each variable and such that
\[
[\tilde g_1,\ldots,\tilde g_{\tilde s}]=\alpha(\tilde\pi(\tilde g_1),\ldots,\tilde\pi(\tilde g_{\tilde s})).
\]
In particular,
\[
[X_0,\ldots,X_0]_{\tilde s}=\alpha(\tilde H,\ldots,\tilde H),
\]
from which it also follows that
\[
\alpha(\tilde H,\ldots,\tilde H)\subset A^{O_{\tilde s}(\tilde m)}.
\]
Proposition \ref{prop:lower.step.in.quotient} therefore follows from the following result.
\begin{prop}\label{prop:grp.in.normal'}
Let $G$ be an $s$-step nilpotent group generated by a $K$-approximate group $A$. Let $\tilde H_1,\ldots,\tilde H_{\tilde s}$ be finite nilpotent groups and suppose that
\[
\alpha:\tilde H_1\times\cdots\times\tilde H_{\tilde s}\to G
\]
is a homomorphism in each variable with the property that
\begin{equation}\label{eq:H.in.A}
\alpha(\tilde H_1,\ldots,\tilde H_{\tilde s})\subset A^M.
\end{equation}
Then there exists a normal subgroup $N$ of $G$ with the property that
\[
\alpha(\tilde H_1,\ldots,\tilde H_{\tilde s})\subset N\subset A^{K^{O_{M,s}(1)}}.
\]
\end{prop}
\begin{proof}
Let $n$ be the smallest integer such that
\[
\gamma(\alpha(\tilde H_1,\ldots,\tilde H_{\tilde s}),G,\ldots,G)=\{1\}
\]
for every commutator form $\gamma$ of weight greater than $n$, noting that $n$ is finite (indeed, at most $s$) by the nilpotence of $G$. We prove by induction on $n$ that there exists a normal subgroup $N$ of $G$ with the property that
\[
\alpha(\tilde H_1,\ldots,\tilde H_{\tilde s})\subset N\subset A^{K^{O_{M,n,s}(1)}},
\]
noting that this is trivial when $n=0$, and that since $n\le s$ it is sufficient to imply the proposition.

For each commutator form $\beta$ of weight $n$, the definition of $n$ implies that
\begin{equation}\label{eq:beta.central}
\beta(\alpha(\tilde H_1,\ldots,\tilde H_{\tilde s}),G,\ldots,G)\subset Z(G).
\end{equation}
Moreover, since there are only finitely many commutator forms of weight $n$, (\ref{eq:H.in.A}) implies that there is some constant $C_{M,n}$ depending only on $M$ and $n$ such that every commutator $\beta$ of weight $n$ satisfies
\begin{equation}\label{eq:beta.alpha.bdd.post.ref}
\beta(\alpha(\tilde H_1,\ldots,\tilde H_{\tilde s}),A^4,\ldots,A^4)\subset A^{C_{M,n}}.
\end{equation}
Combined with (\ref{eq:beta.central}) this implies that every commutator $\beta$ of weight $n$ satisfies
\begin{equation}\label{eq:alpha.central.bdd}
\beta(\alpha(\tilde H_1,\ldots,\tilde H_{\tilde s}),A^4,\ldots,A^4)\subset A^{C_{M,n}}\cap Z(G).
\end{equation}
Lemma \ref{lem:approx.grp.in.[g,g]} implies that $A^{C_{M,n}}\cap Z(G)$ has doubling constant at most $K^{O_{M,n}(1)}$, and so Theorem \ref{thm:ab.freiman} implies that $A^{C_{M,n}}\cap Z(G)$ is contained inside an abelian coset progression $ZQ$, with $Z$ a subgroup satisfying
\begin{equation}\label{eq:centre.coset.prog'}
Z\subset A^{O_{M,n}(1)}
\end{equation}
and $Q$ a progression of rank at most $K^{O_{M,n}(1)}$. In light of (\ref{eq:alpha.central.bdd}), this implies in particular that that every commutator $\beta$ of weight $n$ satisfies
\begin{equation}\label{eq:alpha.coset.prog}
\beta(\alpha(\tilde H_1,\ldots,\tilde H_{\tilde s}),A^4,\ldots,A^4)\subset Z\cdot\langle Q\rangle.
\end{equation}
The fact that $Z$ is central implies in particular that it is normal in $G$, and so we may define the canonical projection $\rho:G\to G/Z$. It then follows from (\ref{eq:alpha.coset.prog}) that that every commutator $\beta$ of weight $n$ satisfies
\begin{equation}\label{eq:beta.alpha.in.bdd.rank}
\beta(\rho(\alpha(\tilde H_1,\ldots,\tilde H_{\tilde s})),\rho(A)^4,\ldots,\rho(A)^4)\subset\langle\rho(Q)\rangle,
\end{equation}
whilst the fact that the progression $Q$ has rank at most $K^{O_{M,n}(1)}$ implies that
\begin{equation}\label{eq:rank.Q.post.ref}
\rank\langle\rho(Q)\rangle\le K^{O_{M,n}(1)}.
\end{equation}
Write $\pi:G/Z\to(G/Z)/\Gamma_2(G/Z)$ for the canonical projection. The image $\pi(\rho(A))$ is a $K$-approximate group, and so Theorem \ref{thm:ab.freiman} implies that there exist an integer
\begin{equation}\label{eq:post.ref.k.bound}
k\le K^{O(1)},
\end{equation}
a subgroup
\begin{equation}\label{eq:post.ref.subgrp.bound}
H\subset\pi(\rho(A)^4),
\end{equation}
elements
\begin{equation}\label{eq:post.ref.elts.bound}
x_2,\ldots,x_k\in\rho(A)^4
\end{equation}
and a vector $L=(L_2,\ldots,L_k)$ of positive integers such that $\pi(\rho(A))$ is contained in a coset progression of the form
\[
H+P(\pi(x_2),\ldots,\pi(x_k);L).
\]
However, $(G/Z)/\Gamma_2(G/Z)$ is generated by $\pi(\rho(A))$, and so in particular we have
\begin{equation}\label{eq:quotient.gen.by.prog}
(G/Z)/\Gamma_2(G/Z)=\langle H\cup\{\pi(x_2),\ldots,\pi(x_k)\}\rangle.
\end{equation}
By (\ref{eq:post.ref.subgrp.bound}) we may fix a subset $Y_1\subset\rho(A)^4$ such that $H=\pi(Y_1)$. For $i=2,\ldots,k$, fix $Y_i:=\{x_i\}$, so that
\begin{equation}\label{eq:Y.in.A}
Y_i\subset\rho(A)^4
\end{equation}
for $i=1,\ldots,k$, and note that (\ref{eq:quotient.gen.by.prog}) implies that
\begin{equation}\label{eq:quotient.gen.by.prog.2}
(G/Z)/\Gamma_2(G/Z)=\langle\pi(Y_1)\cup\cdots\cup\pi(Y_k)\rangle.
\end{equation}
Now for each $\textbf{i}=(i_1,\ldots,i_{n-1})\in[k]^{n-1}$ and each commutator form $\beta$ of weight $n$ define a subgroup
\[
U_{\textbf{i}}(\beta):=\langle\beta(\rho(\alpha(\tilde H_1,\ldots,\tilde H_{\tilde s})),Y_{i_1},\ldots,Y_{i_{n-1}})\rangle.
\]
It follows from (\ref{eq:beta.central}) that $U_{\textbf{i}}(\beta)$ is central in $G/Z$, and so we may define a further central subgroup $U(\beta)$ of $G/Z$ by
\[
U(\beta):=\prod_{\textbf{i}\in[k]^{n-1}}U_{\textbf{i}}(\beta),
\]
and yet another central subgroup $U$ of $G/Z$ by
\[
U:=\prod_{\beta\,:\,|\chi(\beta)|=n}U(\beta).
\]
It follows from Proposition \ref{prop:com.components} and the definition of $n$ that for every commutator form $\beta$ of weight $n$ the map
\[
\begin{array}{ccccc}
\varphi_\beta&:&\tilde H_1\times\cdots\times\tilde H_{\tilde s}\times G/Z\times\cdots\times G/Z&\to& Z(G/Z)\\
&&(\tilde h_1,\ldots,\tilde h_{\tilde s},g_1,\ldots,g_{n-1})&\mapsto&\beta(\rho(\alpha(\tilde h_1,\ldots,\tilde h_{\tilde s})),g_1,\ldots,g_{n-1})
\end{array}
\]
is a homomorphism in each variable, and also that this $\varphi_\beta$ is trivial if there is some $g_i\in\Gamma_2(G/Z)$ (the fact that the range of $\varphi_\beta$ may be taken to be $Z(G/Z)$ follows from (\ref{eq:beta.central})). Lemma~\ref{lem:induced.multi.hom} therefore implies that there is a function
\[
\psi_\beta:\tilde H_1\times\cdots\times\tilde H_{\tilde s}\times(G/Z)/\Gamma_2(G/Z)\times\cdots\times(G/Z)/\Gamma_2(G/Z)\to Z(G/Z)
\]
that is a homomorphism in each variable and satisfies
\begin{equation}\label{eq:beta.to.psi}
\beta(\rho(\alpha(\tilde h_1,\ldots,\tilde h_{\tilde s})),g_1,\ldots,g_n)=\psi_\beta(\tilde h_1,\ldots,\tilde h_{\tilde s},\pi(g_1),\ldots,\pi(g_n)).
\end{equation}
Combined with (\ref{eq:quotient.gen.by.prog.2}), this implies that
\begin{equation}\label{eq:beta.in.prod.U}
\beta(\rho(\alpha(\tilde H_1,\ldots,\tilde H_{\tilde s})),G/Z,\ldots,G/Z)\subset\prod_{\textbf{i}\in[k]^{n-1}}\langle\psi_\beta(\tilde H_1,\ldots,\tilde H_{\tilde s},\pi(Y_{i_1}),\ldots,\pi(Y_{i_{n-1}}))\rangle.
\end{equation}
Now (\ref{eq:beta.to.psi}) also implies that
\begin{equation}\label{eq:beta.alpha.image}
\beta(\rho(\alpha(\tilde H_1,\ldots,\tilde H_{\tilde s})),Y_{i_1},\ldots,Y_{i_{n-1}})=\psi_\beta(\tilde H_1,\ldots,\tilde H_{\tilde s},\pi(Y_{i_1}),\ldots,\pi(Y_{i_{n-1}})),
\end{equation}
and so it follows from (\ref{eq:beta.in.prod.U}) that every commutator form $\beta$ of weight $n$ satisfies
\begin{equation}\label{eq:beta.rho.in.U}
\beta(\rho(\alpha(\tilde H_1,\ldots,\tilde H_{\tilde s})),G/Z,\ldots,G/Z)\subset U.
\end{equation}
Furthermore, (\ref{eq:beta.alpha.in.bdd.rank}), (\ref{eq:Y.in.A}) and (\ref{eq:beta.alpha.image}) combine to imply that
\[
\psi_\beta(\tilde H_1,\ldots,\tilde H_{\tilde s},\pi(Y_{i_1}),\ldots,\pi(Y_{i_{n-1}}))\subset\langle\rho(Q)\rangle,
\]
and so Proposition \ref{prop:img.of.multi.hom} and (\ref{eq:rank.Q.post.ref}) combine to give
\[
\langle\psi_\beta(\tilde H_1,\ldots,\tilde H_{\tilde s},\pi(Y_{i_1}),\ldots,\pi(Y_{i_{n-1}}))\rangle\subset\psi_\beta(\tilde H_1,\ldots,\tilde H_{\tilde s},\pi(Y_{i_1}),\ldots,\pi(Y_{i_{n-1}}))^{K^{O_{M,n}(1)}}.
\]
Combined with (\ref{eq:beta.alpha.bdd.post.ref}), (\ref{eq:Y.in.A}) and (\ref{eq:beta.alpha.image}), this gives
\[
\langle\psi_\beta(\tilde H_1,\ldots,\tilde H_{\tilde s},\pi(Y_{i_1}),\ldots,\pi(Y_{i_{n-1}}))\rangle\subset\rho(A)^{K^{O_{M,n}(1)}}.
\]
It therefore follows from (\ref{eq:beta.alpha.image}) that $U_{\textbf{i}}(\beta)\subset\rho(A)^{K^{O_{M,n}(1)}}$, and hence from (\ref{eq:post.ref.k.bound}) that $U(\beta)\subset\rho(A)^{K^{O_{M,n}(1)}}$. Since there are finitely many commutator forms of weight at most $n$, this implies that
\begin{equation}\label{eq:U.in.A}
U\subset\rho(A)^{K^{O_{M,n}(1)}}.
\end{equation}
Since $U$ is central in $G/Z$, the pullback $N_0:=\rho^{-1}(U)$ is a normal subgroup of $G$, and by (\ref{eq:beta.rho.in.U}) we have
\[
\beta(\alpha(\tilde H_1,\ldots,\tilde H_{\tilde s}),G,\ldots,G)\subset N_0
\]
for every commutator form $\beta$ of weight $n$. In particular, writing $\tau:G\to G/N_0$ for the canonical projection we have
\[
\beta(\tau(\alpha(\tilde H_1,\ldots,\tilde H_{\tilde s})),G/N_0,\ldots,G/N_0)=\{1\}
\]
for every commutator form $\beta$ of weight $n$. The inductive hypothesis therefore implies that there is a normal subgroup $N_1$ of $G/N_0$ satisfying
\begin{equation}\label{eq:N1.in.A}
N_1\subset\tau(A)^{K^{O_{M,n}(1)}}
\end{equation}
and such that $\tau(\alpha(\tilde H_1,\ldots,\tilde H_{\tilde s}))\subset N_1$. Setting $N:=\tau^{-1}(N_1)$, this implies that $\alpha(\tilde H_1,\ldots,\tilde H_{\tilde s})\subset N$. Furthermore, (\ref{eq:centre.coset.prog'}) and (\ref{eq:U.in.A}) imply that $N_0\subset A^{K^{O_{M,n}(1)}}$, and so (\ref{eq:N1.in.A}) implies that $N\subset A^{K^{O_{M,n}(1)}}$.
\end{proof}
We close this section by noting that even if we were to assume Proposition \ref{prop:p.grp.inside.prod.of.lower.step}, and hence reduce to the case in which the approximate group $A$ could be placed inside a bounded product $A_1\cdots A_r$ of approximate groups of lower step, there would still be a significant obstacle to deducing Theorem \ref{thm:freiman.p-group} inductively. This is because applying Theorem \ref{thm:freiman.p-group} to the approximate group $A_1$, say, would produce a subgroup $H$ that was normalised by $A_1$ but not necessarily normalised by the whole of $A$.

It is of significant interest, therefore, to be able to place arbitrary subgroups of $G$ efficiently inside \emph{normal} subgroups of $G$. We therefore take this opportunity to note that applying Proposition \ref{prop:grp.in.normal'} in the case $\tilde{s}=1$ with $\alpha$ the inclusion homomorphism $\alpha:H\hookrightarrow G$ gives the following result.
\begin{prop}\label{prop:grp.in.normal}
Let $G$ be an $s$-step nilpotent group generated by a $K$-approximate group $A$. Let $H\subset A^M$ be a subgroup of $G$. Then there exists a normal subgroup $N$ of $G$ inside $A^{K^{O_{M,s}(1)}}$ with the property that $H\subset N$.
\end{prop}
\section{Conclusion of the general case}\label{sec:conclusion}
In this section we present the final steps of the proof of Theorem \ref{thm:freiman.p-group}, beginning with the following.
\begin{prop}\label{prop:p.grp.inside.prod.of.abelian}
Let $m>0,s,\tilde s\le s$ be integers. Let $G$ be an $s$-step nilpotent group generated by a $K$-approximate group $A$, and let $\tilde A$ be an $\tilde s$-step $\tilde K$-approximate group inside $A^m$. Then there exist an integer $r\le\tilde K^{O_{\tilde s}(1)}$, a normal subgroup $N$ inside $A^{K^{O_{m,s,\tilde s}(1)}}$, and $\tilde K^{O_{\tilde s}(1)}$-approximate groups $A_1,\ldots,A_r\subset\tilde A^{O_{\tilde s}(1)}N$ such that $[A_i,A_i]\subset N$ for all $i$ and such that
\[
\tilde A\subset A_1\cdots A_r.
\]
\end{prop}
\begin{proof}
If $\tilde A$ is abelian then the proposition is trivially true with $r=1$, $A_1=\tilde A$ and $N=\{1\}$. We may therefore assume by induction that $s\ge\tilde s\ge2$ and that the proposition holds for all smaller values of $\tilde s$.

Applying Proposition \ref{prop:p.grp.inside.prod.of.lower.step}, we obtain an integer $r_0\le\tilde K^{O(1)}$, a normal subgroup $N_0$ of $G$ inside $A^{K^{O_{m,s}(1)}}$, and $\tilde K^{O(1)}$-approximate groups $\tilde A_1,\ldots,\tilde A_{r_0}\subset\tilde A^{O(1)}$ such that $[\tilde A_i,\ldots,\tilde A_i]_{\tilde s}\subset N_0$ for all $i$, and such that
\begin{equation}\label{eq:last.post.ref}
\tilde A\subset\tilde A_1\cdots\tilde A_{r_0}.
\end{equation}
Writing $\rho:G\to G/N_0$ for the canonical projection, Lemma \ref{lem:nilp.gens} implies that the $\tilde K^{O(1)}$-approximate groups $\rho(\tilde A_1),\ldots,\rho(\tilde A_{r_0})\subset\rho(\tilde A)^{O(1)}\subset\rho(A)^{O(m)}$ are of step at most $\tilde s-1$, and by (\ref{eq:last.post.ref}) they satisfy
\[
\rho(\tilde A)\subset\rho(\tilde A_1)\cdots\rho(\tilde A_{r_0}).
\]
Since $G/N_0$ is generated by the $K$-approximate group $\rho(A)$, the induction hypothesis therefore gives, for each $i=1,\ldots,r_0$, an integer $r_i\le\tilde K^{O_{\tilde s}(1)}$, a normal subgroup $N_i$ of $G/N_0$ inside $\rho(A)^{K^{O_{m,s,\tilde s}(1)}}$, and $\tilde K^{O_{\tilde s}(1)}$-approximate groups $A_1^{(i)},\ldots,A_{r_i}^{(i)}\subset\rho(\tilde A_i)^{O_{\tilde s}(1)}N_i\subset\rho(\tilde A)^{O_{\tilde s}(1)}N_i$ such that $[A_j^{(i)},A_j^{(i)}]\subset N_i$ for all $j$ and such that
\[
\tilde A_i\subset\rho^{-1}(A_1^{(i)})\cdots\rho^{-1}(A_{r_i}^{(i)}).
\]
It follows from Lemma \ref{lem:pullback} that the $\rho^{-1}(A_j^{(i)})$ are $\tilde K^{O_{\tilde s}(1)}$-approximate groups, and so the proposition is satisfied by taking $N=\rho^{-1}(N_1\cdots N_{r_0})$.
\end{proof}
\begin{remark}When applying Proposition \ref{prop:p.grp.inside.prod.of.abelian} the dependence of various implied constants on $\tilde s$ may of course be absorbed into the dependence on $s$. The dependence on $\tilde s$ is present only to make the proof clearer.
\end{remark}
\begin{proof}[Proof of Theorem \ref{thm:freiman.p-group}]
Let $N$ and $A_1,\ldots,A_r$ be the normal subgroup and approximate groups given by applying Proposition \ref{prop:p.grp.inside.prod.of.abelian} with $\tilde A=A$, noting that
\begin{equation}\label{eq:containment.of.N}
N\subset A^{K^{O_s(1)}}
\end{equation}
and that
\begin{equation}\label{eq:bound.on.r}
r\le K^{O_s(1)}.
\end{equation}
Write $\rho:G\to G/N$ for the canonical projection and note that $\rho(A_1),\ldots,\rho(A_r)$ are abelian $K^{O_s(1)}$-approximate groups. Theorem \ref{thm:ab.freiman} therefore implies that there are subgroups $H_i\subset\rho(A_i)^4\subset\rho(A)^{O_s(1)}$ and abelian progressions
\begin{equation}\label{eq:containment.of.Pi}
P_i\subset\rho(A_i)^{K^{O_s(1)}},
\end{equation}
each of rank at most $K^{O_s(1)}$, such that
\begin{equation}\label{eq:A.in.prod.of.coset.progs}
\rho(A)\subset H_1P_1\cdots H_rP_r.
\end{equation}
Proposition \ref{prop:grp.in.normal} then implies that for each $i=1,\ldots,r$ there is a normal subgroup $N_i$ of $G/N$ that contains $H_i$ and satisfies $N_i\subset\rho(A)^{K^{O_s(1)}}$. By (\ref{eq:bound.on.r}), the normal subgroup $N_1\cdots N_r$ also satisfies
\[
N_1\cdots N_r\subset\rho(A)^{K^{O_s(1)}},
\]
and so by (\ref{eq:containment.of.N}) the normal subgroup $H=\rho^{-1}(N_1\cdots N_r)$ satisfies
\[
H\subset A^{K^{O_s(1)}}.
\]
Write $\tau:G\to G/H$ for the canonical projection. The bound on the ranks of the $P_i$, combined with (\ref{eq:bound.on.r}) and (\ref{eq:A.in.prod.of.coset.progs}), implies that $\tau(A)$ is contained in an ordered progression $P$ of rank at most $K^{O_s(1)}$, and by (\ref{eq:bound.on.r}) and (\ref{eq:containment.of.Pi}) we have $P\subset\tau(A)^{K^{O_s(1)}}$. The result then follows from (\ref{eq:equiv.of.progs}).
\end{proof}
\appendix
\section{Commutators and progressions}\label{sec:commutators}
In this appendix we develop enough theory about commutators to be able to define a nilpotent progression, and hence to understand the statement of Theorem \ref{thm:freiman.p-group}. Much of this theory is already in the literature, but some of our arguments rely on an inspection of existing proofs, not just on the results, and in setting up notation for these arguments we unavoidably end up repeating much of this standard material.

An argument from \cite[\S11.1]{hall} shows that commutators in the generators of a nilpotent group satisfy the following rather nice property.
\begin{lemma}\label{lem:collecting}
Let $x_1,\ldots,x_r$ be elements in an $s$-step nilpotent group and let $y_1\prec\ldots\prec y_k$ be a complete list of commutators of total weight at most $s$ in the $x_i$, with $y_i=x_i$ for $i=1,\ldots,r$. Then every element of the group generated by the $x_i$ can be expressed in the form
\begin{equation}\label{eq:collected}
y_1^{l_1}y_2^{l_2}\cdots y_k^{l_k},
\end{equation}
with $l_i\in\Z$.
\end{lemma}
The proof of Lemma~\ref{lem:collecting} is in the form of an algorithm, called the \emph{collecting process} and defined below, that converts a finite string involving the elements $x_i$ into the form (\ref{eq:collected}). The collecting process rests on, and is motivated by, the following commutator identities, which are \cite[(11.1.5, 11.1.8, 11.1.9, 11.1.11)]{hall}.
\begin{equation}\label{eq:collecting.1}
vu=uv[v,u]
\end{equation}
\begin{equation}\label{eq:collecting.2}
v^{-1}u=u[v,u]^{-1}v^{-1}
\end{equation}
\begin{equation}\label{eq:collecting.3}
vu^{-1}=u^{-1}vv_2v_4\cdots v_5^{-1}v_3^{-1}v_1^{-1}\qquad(v_1:=[v,u];\,v_{i+1}:=[v_i,u])
\end{equation}
\begin{equation}\label{eq:collecting.4}
v^{-1}u^{-1}=u^{-1}v_1v_3v_5\cdots v_4^{-1}v_2^{-1}v^{-1}\qquad(v_1:=[v,u];\,v_{i+1}:=[v_i,u])
\end{equation}
\begin{definition}[The collecting process {\cite[\S11.1]{hall}}]
Let $a:=a_1\cdots a_n$ be a string of commutators, and their inverses, involving elements $x_1,\ldots,x_r$ of an $s$-step nilpotent group. We say that $a$ is in \emph{collected form} if $a_1\prec\ldots\prec a_n$. For $m\in[0,n]$ we say that $a$ has \emph{collected part} $a_1\cdots a_m$ if $a_1\prec\ldots\prec a_m$ and $a_m\prec a_i$ for $i>m$ but $a_{m+1}\succ a_j$ for some $j>m+1$. The string $a_{m+1}\cdots a_n$ is called the \emph{uncollected part} of $a$.

We define a \emph{collecting operator} $\mathcal{C}$ on the set of strings of commutators as follows. If $a$ is in collected form then we simply define $\mathcal{C}(a):=a$. If $a$ is not in collected form then write $\beta$ for the commutator earliest with respect to the order $\prec$ such that $\beta^{\pm1}$ appears in the uncollected part of $a$, and suppose that $a_j$ is the leftmost of the $a_i$ in the uncollected part that is equal to $\beta^{\pm1}$. Write $\alpha$ for the commutator satisfying $a_{j-1}=\alpha^{\pm1}$, and define $\alpha_1:=[\alpha,\beta]$ and $\alpha_{i+1}:=[\alpha_i,\beta]$ for $i=2,\ldots,s-1$. Define $\alpha_i$ to be trivial for $i\ge s$.

If $a_j=\beta$ and $a_{j-1}=\alpha$ then
\[
\mathcal{C}(a):=a_1\cdots a_{j-2}\beta\alpha[\alpha,\beta]a_{j+1}\cdots a_n,
\]
and we say that $\mathcal{C}$ performs a \emph{transformation of type 1} on $a$. If $a_j=\beta$ and $a_{j-1}=\alpha^{-1}$ then
\[
\mathcal{C}(a):=a_1\cdots a_{j-2}\beta[\alpha,\beta]^{-1}\alpha^{-1}a_{j+1}\cdots a_n,
\]
and we say that $\mathcal{C}$ performs a \emph{transformation of type 2} on $a$. If $a_j=\beta^{-1}$ and $a_{j-1}=\alpha$ then
\[
\mathcal{C}(a):=a_1\cdots a_{j-2}\beta^{-1}\alpha\alpha_2\alpha_4\cdots\alpha_5^{-1}\alpha_3^{-1}\alpha_1^{-1}a_{j+1}\cdots a_n,
\]
and we say that $\mathcal{C}$ performs a \emph{transformation of type 3} on $a$. If $a_j=\beta^{-1}$ and $a_{j-1}=\alpha^{-1}$ then
\[
\mathcal{C}(a):=a_1\cdots a_{j-2}\beta^{-1}\alpha_1\alpha_3\alpha_5\cdots\alpha_4^{-1}\alpha_2^{-1}\alpha^{-1}a_{j+1}\cdots a_n,
\]
and we say that $\mathcal{C}$ performs a \emph{transformation of type 4} on $a$. The collecting process is then the sequence $a,\mathcal{C}(a),\mathcal{C}^2(a),\ldots$.

Finally, we will often talk about the set of \emph{copies} of a commutator in the collecting process. Each $a_i$ in the original string $a$ is said to be a copy of the commutator $\gamma$ if $a_i=\gamma$, or a copy of $\gamma^{-1}$ if $a_i=\gamma^{-1}$. The collecting operator $\mathcal{C}$ is then thought of as permuting the existing copies of the commutators and of creating a new copy of $[\alpha,\beta]$ in the case of a transformation of type 1; a new copy of $[\alpha,\beta]^{-1}$ in the case of a transformation of type 2; new copies of $\alpha_2,\alpha_4\ldots$ and $\alpha_1^{-1}\alpha_3^{-1}\ldots$ in the case of a transformation of type 3; and new copies of $\alpha_1,\alpha_3\ldots$ and $\alpha_2^{-1}\alpha_4^{-1}\ldots$ in the case of a transformation of type 4.
\end{definition}
\begin{remark}
The collecting process can be defined for \emph{positive} strings of commutators (i.e. those involving only commutators, not their inverses) without the assumption that the $x_i$ are elements of a nilpotent group; see \cite[\S11.1]{hall} for details.
\end{remark}
\begin{proof}[Proof of Lemma~\ref{lem:collecting}]
An element in the group generated by the $x_i$ is, by definition, equal in the group to some string $a$ in the $x_i$. It is clear (see \cite[\S11.1]{hall} for details) that the sequence $a,\mathcal{C}(a),\mathcal{C}^2(a),\ldots$ of strings in the collecting process is eventually constant and in collected form (modulo commutators of weight greater than $s$). Finally, identities (\ref{eq:collecting.1}), (\ref{eq:collecting.2}), (\ref{eq:collecting.3}) and (\ref{eq:collecting.4}) show that the sequence $a,\mathcal{C}(a),\mathcal{C}^2(a),\ldots$ is constant as a sequence of group elements.
\end{proof}
Only certain commutators will arise during the collecting process. For example, $x_2$ never needs to be moved to the left of $x_1$, and so the exponent of the commutator $[x_1,x_2]$ in expression (\ref{eq:collected}) will always be zero. It thus makes sense to introduce the following definition.
\begin{definition}[Basic commutators]\label{def:basic.coms}
The commutators that can arise when the collecting process is applied to a string consisting only of $x_1,\ldots,x_r$, and can hence have non-zero exponents in expression (\ref{eq:collected}), are called \emph{basic commutators}. We shall denote them $c_1\prec\ldots\prec c_t$.
\end{definition}
A more explicit definition of basic commutators is given in \cite[\S11.1]{hall}, but it is formulated precisely so as to be equivalent to Definition~\ref{def:basic.coms}.

We are now in a position to make another definition, first seen in \cite{tor.free.nilp}.
\begin{definition}[Nilpotent progression]\label{def:nilp.prog}
Let $G$ be a nilpotent group. Let $x_1,\ldots,x_r$ be elements of $G$, and let $c_1,\ldots,c_t$ be the list of basic commutators in the $x_i$. Let $L=(L_1,\ldots,L_r)$ be a vector of positive integers. Then the \emph{nilpotent progression} on $x_1,\ldots,x_r$ with side lengths $L_1,\ldots,L_r$ is defined to be the set
\[
P(x_1,\ldots,x_r;L):=\{c_1^{l_1}\cdots c_t^{l_t}:|l_i|\le L^{\chi(c_i)}\}.
\]
\end{definition}
The following results give a straightforward relationship between nilprogressions and nilpotent progressions.
\begin{lemma}\label{lem:quant.collect}
Let $x_1,\ldots,x_r$ be elements of a nilpotent group and let $L=(L_1,\ldots,L_r)$ be a vector of positive integers. Let $y$ be a string of $x_i$ and their inverses featuring $p_j$ copies of $x_j$ and $n_j$ copies of $x_j^{-1}$, with $p_j+n_j\le L_j$. Then for each $i=1,\ldots,t$, applying the collecting process to $y$ results in at most $L^{\chi(c_i)}$ copies of $c_i$ and $c_i^{-1}$ between them.
\end{lemma}
\begin{proof}
The collecting process never creates any new copies of $x_i$ or $x_i^{-1}$, and so the lemma certainly holds whenever $c_i$ has total weight 1. We may therefore, by induction, prove the lemma for basic commutators of a given total weight
\begin{equation}\label{eq:quant.collect.2}
\omega>1
\end{equation}
under the assumption that it holds for all commutators of total weight less than $\omega$.

It follows from (\ref{eq:quant.collect.2}) that any commutator of weight $\omega$ arising from the collecting process is of the form $[c_i,c_j]^{\pm1}$ for some basic commutators $c_i,c_j$ of total weight less than $\omega$. We will prove the lemma by defining an injection $f$ from the set of copies of $[c_i,c_j]^{\pm1}$ to the set of pairs of copies of $c_i^{\pm1}$ and $c_j^{\pm1}$. This is sufficient as by induction the number of such pairs is at most $L^{\chi(c_i)}L^{\chi(c_j)}=L^{\chi([c_i,c_j])}$.

If a copy $z$ of $[c_i,c_j]^{\pm1}$ arose from a collecting transformation of type 1 or 2 then this can only have been as a result of interchanging a copy $a$ of $c_i^{\pm1}$ and a copy $b$ of $c_j$. In this case simply define $f(z)=(a,b)$.

If $z$ arose as from a collecting transformation of type 3 or 4 then it must have arisen from interchanging a copy $b'$ of $c_j^{-1}$ and a copy of some commutator $c_k$. Writing $v_0:=c_k$ and $v_{q+1}=[v_q,c_j]$, it must therefore be the case that $c_i$ is equal to some $v_q$. If $q=0$ then $z$ arose from interchanging a copy $a'$ of $c_i$ with $b'$, and we define $f(z)=(a',b')$. If $q>0$ then exactly one copy $a''$ of $v_q$ will also have arisen as a result of the same transformation as that producing $z$. In this case we define $f(z)=(a'',b')$.

Now that we have defined $f$ it remains to show that it is an injection. A given pair of copies of $c_i^{\pm1}$ and $c_j^{\pm1}$ will be interchanged at most once during the collection process, and so if two distinct copies $z,z'$ of $[c_i,c_j]$ both arose from interchanging copies of $c_i$ and $c_j$ (which is to say as a result of a transformation of type 1 or 2, or of type 3 or 4 in the case $q=0$) then we certainly have $f(z)\ne f(z')$. Furthermore, a given transformation of type 3 or 4 produces at most one copy of any commutator $c_l$, and so if $z$ and $z'$ both arose from transformations of type 3 or 4 in the case $q>0$ then we also have $f(z)\ne f(z')$.

Finally, suppose a copy $z$ of $[c_i,c_j]^{\pm1}$ arose from a transformation of type 3 or 4 in the case $q>0$. It follows from (\ref{eq:collecting.3}) and (\ref{eq:collecting.4}) that the copy $a''$ of $c_i=v_q$ produced as a result of the transformation producing $z$ will already be to the right of $b'$, and so $a''$ and $b'$ will never have to be interchanged and so no $z'$ arising from a transformation of type 1 or 2, or of type 3 or 4 in the case $q=0$, will have $f(z')=(a'',b')$.
\end{proof}
\begin{corollary}\label{cor:basic.collected}
Let $x_1,\ldots,x_r$ be elements of a nilpotent group and let $L=(L_1,\ldots,L_r)$ be a vector of positive integers. Let $y$ be a string of $x_i$ and their inverses featuring $p_j$ copies of $x_j$ and $n_j$ copies of $x_j^{-1}$, with $p_j+n_j\le L_j$. Then $y$ can be expressed in the form
\begin{equation}\label{eq:basic.collected}
c_1^{l_1}c_2^{l_2}\cdots c_t^{l_t},
\end{equation}
with $l_i=p_i-n_i$ for $i=1,\ldots,r$ and $|l_i|\le L^{\chi(c_i)}$ for every $i$.
\end{corollary}
\begin{corollary}\label{cor:nilprog.in.nilp.prog}
Let $x_1,\ldots,x_r$ be elements of a nilpotent group and let $L=(L_1,\ldots,L_r)$ be a vector of positive integers. Then
\[
P^*(x_1,\ldots,x_r;L)\subset P(x_1,\ldots,x_r;L).
\]
\end{corollary}
Here, of course, $P^*$ denotes a nilprogression, as in Definition~\ref{def:nilprog}, while $P$ denotes a nilpotent progression, as in Definition~\ref{def:nilp.prog}.
\section{Decompositions of commutators}\label{sec:commutators.2}
It is a straightforward exercise to show that in a 2-step nilpotent group we have the identity
\[
[xx',yy']=[x,y][x',y][x,y'][x',y'].
\]
The aim of this appendix is to develop a more general calculus for expressing commutators of products as products of commutators.

In order to achieve this we will need to consider commutators involving $x_1,\ldots,x_r$ in more detail than in Appendix \ref{sec:commutators}. For instance, given a commutator $\alpha$ it will be helpful to be able to refer to the lower-weight components that were used to construct $\alpha$, such as the components $x_1,x_2,x_3,[x_2,x_3]$ in the case $\alpha=[x_1,[x_2,x_3]]$. We formalise this notion as follows.
\begin{definition}[Arguments and components of a commutator]
To every formal commutator $\alpha$ in the letters $x_1,\ldots,x_r$ we assign a set $C(\alpha)$ of \emph{components}, defined recursively as follows. For each $i$ define $C(x_i)$ to be the singleton $\{x_i\}$ and, given two formal commutators $\alpha,\alpha'$ in the $x_i$, define $C([\alpha,\alpha']):=C(\alpha)\cup C(\alpha')\cup\{[\alpha,\alpha']\}$. We shall denote by $\alpha\sqsubset\beta$ or $\beta\sqsupset\alpha$ the fact that $\alpha$ is a component of $\beta$. The letters $x_{i_j}$ featuring in a formal commutator $\alpha$ will be called the \emph{arguments} of $\alpha$.
\end{definition}
The principal results are the following.
\begin{prop}\label{prop:com.components}
Let $\Gamma$ be a nilpotent group and for each $j=1,\ldots,r$ let $x^{(j)}_1,\ldots,x^{(j)}_{m_j}$ be elements of $\Gamma$. Suppose that $\alpha$ is a commutator form of weight $r$. Then there exist commutators $\eta_1,\ldots,\eta_t$ in the $x^{(j)}_i$ with the following properties.
\begin{enumerate}
\renewcommand{\labelenumi}{(\roman{enumi})}
\item We have $\alpha(\prod_{i=1}^{m_1}x^{(1)}_i,\ldots,\prod_{i=1}^{m_r}x^{(r)}_i)=\eta_1\cdots\eta_t$.
\item For each tuple $(x^{(1)}_{i_1},\ldots,x^{(r)}_{i_r})$ there is some $\eta_l$ equal to $\alpha(x^{(1)}_{i_1},\ldots,x^{(r)}_{i_r})$.
\item For every $l\in[1,t]$ and every $j\in[1,r]$ there is at least one $i\in[1,m_j]$ with $x^{(j)}_i\sqsubset\eta_l$.
\item The $\eta_l$ are all distinct as formal commutators in the $x^{(j)}_i$.
\item If $\eta_l$ is not of the form $\alpha(x^{(1)}_{i_1},\ldots,x^{(r)}_{i_r})$ then $\eta_l$ has weight greater than $r$.
\end{enumerate}
\end{prop}
\begin{prop}\label{prop:com.powers}
Let $\alpha$ be a commutator form of weight $r$, let $x_1,\ldots,x_r$ be elements of an $s$-step nilpotent group $\Gamma$, and let $(l_1,\ldots,l_r)$ be a vector of positive integers. Let $\xi_1,\ldots,\xi_k$ be the complete list, in ascending order with respect to $\prec$, of commutators in the $x_i$ of weight greater than $r$, and write $\alpha(x_1,\ldots,x_r)$, $\zeta_1,\ldots,\zeta_t$ for the ordered list of basic commutators in $\alpha(x_1,\ldots,x_r)$, $\xi_1,\ldots,\xi_k$. Then for each $i=1,\ldots,t$ there exists an integer $m_i$ with $|m_i|\le l^{\chi(\zeta_i)}$, where $\chi(\zeta_i)$ is the weight vector of $\zeta_i$ as a commutator in the $x_i$, such that
\[
\alpha(x_1^{l_1},\ldots,x_r^{l_r})=\alpha(x_1,\ldots,x_r)^{l_1\cdots l_r}\zeta_1^{m_1}\cdots\zeta_t^{m_t}
\]
in $\Gamma$.
\end{prop}
\begin{remark}The commutators $\zeta_i$ may not be distinct as commutators in the $x_i$, but they are distinct as formal commutators in the letters $\xi_i$.
\end{remark}
In proving these results we will repeatedly make use of certain commutator identities. Let $u,v,w$ be elements of a group.  We start by recalling the trivial identity (\ref{eq:collecting.1}), which was
\begin{equation}\label{eq:com.iden.0}
vu=uv[v,u],
\end{equation}
and note that repeated application of it implies that
\begin{displaymath}
\begin{array}{rcl}
uvw[w,uv] & = & wuv\\
          & = & uw[w,u]v\\
          & = & uwv[w,u][[w,u],v]\\
          & = & uvw[w,v][w,u][[w,u],v].
\end{array}
\end{displaymath}
Combined with another trivial identity, $[a,b]^{-1}=[b,a]$, this implies that
\begin{equation}\label{eq:com.iden.1}
[u,w][v,w]=[[w,u],v][uv,w]
\end{equation}
and
\begin{equation}\label{eq:com.iden.2}
[w,v][w,u]=[w,uv][v,[w,u]].
\end{equation}
\begin{lemma}\label{lem:higher.comms.different}
Let $x_1,\ldots,x_m,y_1,\ldots,y_n$ be symbols. Suppose that $\xi_1,\ldots,\xi_r$ are distinct formal commutators in the $x_i$ and that $\zeta_1,\ldots,\zeta_t$ are distinct formal commutators in the $y_i$. Suppose further that $\beta,\beta'$ are formal commutators in the $\xi_i$ and the $\zeta_i$ that have no components of the form $[\xi_i,\xi_j]$ or $[\zeta_i,\zeta_j]$. Now suppose that $\beta$ and $\beta'$ are equal when viewed as commutators in the $x_i$ and $y_i$. Then $\beta$ and $\beta'$ are equal as commutators in the $\xi_i$ and $\zeta_i$.
\end{lemma}
\begin{proof}
Viewing $\beta$ as a commutator in the $x_i$ and $y_i$, write $\eta_1,\ldots,\eta_u$ for the list of components of $\beta$ that are maximal with respect to the property of having no $y_i$ as an argument, and write $\gamma_1,\ldots,\gamma_v$ for the list of components that are maximal with respect to the property of having no $x_i$ as an argument. Express $\beta$ as a commutator in the $\eta_i$ and the $\gamma_i$. Thus, for example, if
\[
\beta=[[[x_1,x_2],y_1],[x_1,x_2]]
\]
then $\eta_1=\eta_2=[x_1,x_2]$ and $\gamma_1=y_1$ and we would have
\[
\beta=[[\eta_1,\gamma_1],\eta_2].
\]
Note that the $\eta_i,\gamma_i$ and the form of $\beta$ in terms of the $\eta_i$ and $\gamma_i$ are uniquely determined by the form of $\beta$ in terms of the $x_i$ and $y_i$.

Viewing $\beta$ as a commutator in the $\xi_i$ and $\zeta_i$ once more, the only way that an $\eta_l$ can have arisen is as a commutator in the $\xi_i$. In fact, since $\beta$ contains no component of the form $[\xi_i,\xi_j]$, the component $\eta_l$ must have arisen as some $\xi_i$. Similarly, each component $\gamma_l$ must have arisen as some $\zeta_i$. Since the $\xi_i$ and $\zeta_i$ are distinct, these $\xi_i$ and $\zeta_i$ are uniquely determined, and so the form of $\beta$ as a commutator in the $\xi_i$ and $\zeta_i$ is uniquely determined by the $\eta_i$ and $\gamma_i$ and by the form of $\beta$ with respect to the $\eta_i$ and the $\gamma_i$.

This in turn was uniquely determined by the form of $\beta$ as a commutator in the $x_i$ and $y_i$, and so the proof is complete.
\end{proof}
\begin{lemma}\label{lem:comm.decomp}
Let $x_1,\ldots,x_m,y_1,\ldots,y_n$ be elements of a nilpotent group $\Gamma$. Then there is a finite list $\eta_1,\ldots,\eta_r$ of commutators in the $x_i$ and $y_j$ satisfying the following properties.
\begin{enumerate}
\renewcommand{\labelenumi}{(\roman{enumi})}
\item We have $[x_1\cdots x_m,y_1\cdots y_n]=\eta_1\cdots\eta_r$.
\item For each pair $x_i,y_{i'}$ there is some $\eta_j$ equal to $[x_i,y_{i'}]$.
\item Every $\eta_j$ has at least one $x_i$ and at least one $y_{i'}$ amongst its arguments.
\item The $\eta_j$ are all distinct as formal commutators in the $x_i,y_{i'}$.
\item No $\eta_j$ has a component of the form $[x_i,x_{i'}]$ or $[y_i,y_{i'}]$.
\item If $\eta_j$ is not of the form $[x_i,y_{i'}]$ then $\eta$ has weight greater than 2.
\end{enumerate}
\end{lemma}
\begin{proof}
For each element $g\in\Gamma$ define
\[
\theta(g):=\sup\{i:g\in\Gamma_i\}.
\]
Write $s$ for the nilpotency class of $\Gamma$ and note that if $\min_i\theta(x_i)+\min_j\theta(y_j)>s$ then $[x_1\cdots x_m,y_1\cdots y_n]$ and the $[x_i,y_j]$ are all equal to the identity element, and so the result is trivial. By induction, it is therefore sufficient to fix $s$ and $d\le s$; to assume that
\begin{equation}\label{eq:comm.decomp.1}
\min_i\theta(x_i)+\min_j\theta(y_j)=d;
\end{equation}
and to assume that the result holds for elements of $s$-step nilpotent groups satisfying
\[
\min_i\theta(x_i)+\min_j\theta(y_j)>d.
\]
Such a statement is trivial in the case $m+n\le2$, and so by another induction we may assume also that $m+n\ge3$ and that the result holds under condition (\ref{eq:comm.decomp.1}) for all smaller values of $m+n$.

The assumption that $m+n\ge3$ implies in particular that at least one of $m$ and $n$ is at least 2. We shall assume that $n\ge2$; the argument in the case $m\ge2$ is essentially identical and is left to the reader. Thanks to this assumption we may apply (\ref{eq:com.iden.2}) and conclude that
\begin{equation}\label{eq:comm.decomp.2}
[x_1\cdots x_m,y_1\cdots y_n]=[x_1\cdots x_m,y_2\cdots y_n]\,[x_1\cdots x_m,y_1]\,[[x_1\cdots x_m,y_1],y_2\cdots y_n].
\end{equation}
If the first induction hypothesis does not apply to $[x_1\cdots x_m,y_2\cdots y_n]$ then the second certainly does, and so we may assume that there are commutators $\eta_1,\ldots,\eta_q$ in $x_1,\ldots,x_m,y_2,\ldots,y_n$ such that the following conditions hold.
\begin{enumerate}
\item\label{eq:comm.decomp.3}We have $[x_1\cdots x_m,y_2\cdots y_n]=\eta_1\cdots\eta_q$.
\item\label{eq:comm.decomp.4}For each pair $x_i,y_{i'}$ with $i'\ge2$ there is some $\eta_j$ equal to $[x_i,y_{i'}]$.
\item\label{eq:comm.decomp.5}Every $\eta_j$ has at least one $x_i$ and at least one $y_{i'}$ with $i'\ge2$ amongst its arguments.
\item\label{eq:comm.decomp.6}The $\eta_j$ are all distinct as formal commutators in the $x_i,y_{i'}$.
\item\label{eq:comm.decomp.6.1}No $\eta_j$ has a component of the form $[x_i,x_{i'}]$ or $[y_i,y_{i'}]$.
\item\label{eq:comm.decomp.6.3}If $\eta_j$ is not of the form $[x_i,y_{i'}]$ then $\eta$ has weight greater than 2.
\end{enumerate}
Similarly, we have commutators $\xi_1,\ldots,\xi_r$ in the $x_i$ and $y_1$ such that
\begin{enumerate}\setcounter{enumi}{6}
\item\label{eq:comm.decomp.7}We have $[x_1\cdots x_m,y_1]=\xi_1\cdots\xi_r$.
\item\label{eq:comm.decomp.8}For each $x_i$ there is some $\xi_j$ equal to $[x_i,y_1]$.
\item\label{eq:comm.decomp.9}Every $\xi_j$ has $y_1$ and at least one $x_i$ amongst its arguments.
\item\label{eq:comm.decomp.10}The $\xi_j$ are all distinct as formal commutators in the $x_i,y_{i'}$.
\item\label{eq:comm.decomp.10.1}No $\xi_j$ has a component of the form $[x_i,x_{i'}]$.
\item\label{eq:comm.decomp.10.2}If $\xi_j$ is not of the form $[x_i,y_1]$ then $\eta$ has weight greater than 2.
\end{enumerate}
It follows from condition (\ref{eq:comm.decomp.7}) that we may write
\begin{equation}\label{eq:comm.decomp.11}
[[x_1\cdots x_m,y_1],y_2\cdots y_n]=[\xi_1\cdots\xi_r,y_2\cdots y_n],
\end{equation}
while condition (\ref{eq:comm.decomp.9}) implies that for evey $j$ we have 
\[
\theta(\xi_j)\ge\min_i\theta(x_i)+\theta(y_1).
\]
This implies in particular that
\[
\min_i\theta(\xi_i)+\min_{j\in[2,n]}\theta(y_j)>\min_i\theta(x_i)+\min_j\theta(y_j),
\]
and hence by (\ref{eq:comm.decomp.1}) that
\[
\min_i\theta(\xi_i)+\min_{j\in[2,n]}\theta(y_j)>d.
\]
The first induction hypothesis therefore applies, and so we may assume that there exist commutators $\zeta_1,\ldots,\zeta_t$ in the $\xi_i,y_j$ satisfying the following conditions.
\begin{enumerate}\setcounter{enumi}{12}
\item\label{eq:comm.decomp.12}We have $[\xi_1\cdots\xi_r,y_2\cdots y_n]=\zeta_1\cdots\zeta_t$.
\item\label{eq:comm.decomp.13}Every $\zeta_j$ has at least one $\xi_i$ and at least one $y_{i'}$ with $i'\ge2$ amongst its arguments.
\item\label{eq:comm.decomp.14}The $\zeta_j$ are all distinct as formal commutators in the $\xi_i,y_{i'}$.
\item\label{eq:comm.decomp.15}No $\zeta_j$ has a component of the form $[\xi_i,\xi_{i'}]$ or $[y_i,y_{i'}]$.
\end{enumerate}
It follows from (\ref{eq:comm.decomp.2}) and (\ref{eq:comm.decomp.11}) and conditions (\ref{eq:comm.decomp.3}), (\ref{eq:comm.decomp.7})  and (\ref{eq:comm.decomp.12}) that
\[
[x_1\cdots x_m,y_1\cdots y_n]=\eta_1\cdots\eta_q\xi_1\cdots\xi_r\zeta_1\cdots\zeta_t,
\]
and so the $\eta_i,\xi_i,\zeta_i$ satisfy condition (i) of the lemma. By conditions (\ref{eq:comm.decomp.4}) and (\ref{eq:comm.decomp.8}) they satisfy condition (ii) of the lemma, and by conditions (\ref{eq:comm.decomp.5}), (\ref{eq:comm.decomp.9}) and (\ref{eq:comm.decomp.13}) they satisfy condition (iii). It follows from conditions (\ref{eq:comm.decomp.6.1}), (\ref{eq:comm.decomp.10.1}), (\ref{eq:comm.decomp.10.2}) and (\ref{eq:comm.decomp.15}) that they satisfy condition (v) of the lemma, and from conditions (\ref{eq:comm.decomp.6.3}), (\ref{eq:comm.decomp.10.2}) and (\ref{eq:comm.decomp.13}) that they satisfy condition (vi).

It remains to see that the $\eta_i,\xi_i,\zeta_i$ satisfy condition (iv) of the lemma. No two $\eta_i$ can be equal as formal commutators in the $x_i,y_j$ thanks to condition (\ref{eq:comm.decomp.6}), and the same is true of the $\xi_i$ by condition (\ref{eq:comm.decomp.10}). It follows from Lemma~\ref{lem:higher.comms.different} and from conditions (\ref{eq:comm.decomp.14}) and (\ref{eq:comm.decomp.15}) that the same is also true of the $\zeta_i$.

It follows from conditions (\ref{eq:comm.decomp.8}) and (\ref{eq:comm.decomp.13}) that each $\xi_i$ and each $\zeta_j$ has $y_1$ as an argument, but none of the $\eta_i$ has $y_1$ as an argument, and so no $\eta_i$ can be equal to a $\xi_j$ or a $\zeta_j$ as a formal commutator in the $x_i,y_j$. Moreover, $y_1$ is the only of the $y_j$ to feature as an argument of a $\xi_i$, whereas it follows from condition (\ref{eq:comm.decomp.13}) that each $\zeta_i$ has at least one $y_j$ other than $y_1$ as an argument, and so no $\zeta_i$ can equal be equal to a $\xi_j$ as a formal commutator in the $x_i,y_j$. Thus $\eta_i,\xi_i,\zeta_i$ satisfy condition (iv) of the lemma, and so the proof is complete.
\end{proof}
\begin{proof}[Proof of Proposition \ref{prop:com.components}]
Abusing notation slightly, we shall abbreviate
\[
\alpha=\alpha(\textstyle\prod_{i=1}^{m_1}x^{(1)}_i,\ldots,\prod_{i=1}^{m_r}x^{(r)}_i).
\]
If $\alpha$ is of weight 1 then the result is trivial, and so we may assume that $r>1$ and, by induction, that the result holds for all commutator forms of weight less than $r$. Since $r>1$ we may write
\[
\alpha=\left[\alpha_1(\textstyle\prod_{i=1}^{m_1}x^{(1)}_i,\ldots,\prod_{i=1}^{m_d}x^{(d)}_i),\alpha_2(\prod_{i=1}^{m_{d+1}}x^{(d+1)}_i,\ldots,\prod_{i=1}^{m_r}x^{(r)}_i)\right]
\]
for commutator forms $\alpha_1$ and $\alpha_2$ of weight $d$ and $d':=r-d$, respectively, with
\begin{equation}\label{eq:com.components.1}
1\le d,d'<r.
\end{equation}
For notational convenience let us relabel $x^{(d+1)}_i,\ldots,x^{(r)}_i$ as $y^{(1)}_i,\ldots,y^{(d')}_i$, and $m_{d+1},\ldots,m_r$ as $n_1,\ldots,n_{d'}$; thus
\begin{equation}\label{eq:com.components.1.5}
\alpha=\left[\alpha_1(\textstyle\prod_{i=1}^{m_1}x^{(1)}_i,\ldots,\prod_{i=1}^{m_d}x^{(d)}_i),\alpha_2(\prod_{i=1}^{n_1}y^{(1)}_i,\ldots,\prod_{i=1}^{n_{d'}}y^{(d')}_i)\right].
\end{equation}
By (\ref{eq:com.components.1}) and the induction hypothesis there exist commutators $\xi_1,\ldots,\xi_k$ in the $x^{(j)}_i$ and commutators $\zeta_1,\ldots,\zeta_q$ in the $y^{(j)}_i$ satisfying the following conditions.
\begin{enumerate}
\item\label{eq:com.components.2} We have $\alpha_1(\textstyle\prod_{i=1}^{m_1}x^{(1)}_i,\ldots,\prod_{i=1}^{m_d}x^{(d)}_i)=\xi_1\cdots\xi_k$.
\item\label{eq:com.components.3} We have $\alpha_2(\textstyle\prod_{i=1}^{n_1}y^{(1)}_i,\ldots,\prod_{i=1}^{n_{d'}}y^{(d')}_i)=\zeta_1\cdots\zeta_q$.
\item\label{eq:com.components.4} For each tuple $(x^{(1)}_{i_1},\ldots,x^{(r)}_{i_d})$ there is some $\xi_l$ equal to $\alpha_1(x^{(1)}_{i_1},\ldots,x^{(d)}_{i_d})$.
\item\label{eq:com.components.5} For each tuple $(y^{(1)}_{i_1},\ldots,y^{(d')}_{i_{d'}})$ there is some $\zeta_l$ equal to $\alpha_2(y^{(1)}_{i_1},\ldots,y^{(d')}_{i_{d'}})$.
\item\label{eq:com.components.6} For every $l\in[1,k]$ and every $j\in[1,d]$ there is at least one $i\in[1,m_j]$ with $x^{(j)}_i\sqsubset\xi_l$.
\item\label{eq:com.components.7} For every $l\in[1,q]$ and every $j\in[1,d']$ there is at least one $i\in[1,n_j]$ with $y^{(j)}_i\sqsubset\zeta_l$.
\item\label{eq:com.components.8} The $\xi_l$ are all distinct as formal commutators in the $x^{(j)}_i$.
\item\label{eq:com.components.9} The $\zeta_l$ are all distinct as formal commutators in the $y^{(j)}_i$.
\end{enumerate}
By (\ref{eq:com.components.1.5}) and conditions (\ref{eq:com.components.2}) and (\ref{eq:com.components.3}) we have
\[
\alpha=[\xi_1\cdots\xi_k,\zeta_1\cdots\zeta_q],
\]
and so by Lemma~\ref{lem:comm.decomp} we may assume that there exist commutators $\eta_1,\ldots,\eta_t$ in the $\xi_i$ and $\zeta_i$ satisfying the following conditions.

\begin{enumerate}\setcounter{enumi}{8}
\item\label{eq:com.components.10} We have $\alpha=\eta_1\cdots\eta_t$.
\item\label{eq:com.components.11} For each pair $\xi_i,\zeta_{i'}$ there is some $\eta_j$ equal to $[\xi_i,\zeta_{i'}]$.
\item\label{eq:com.components.12} Every $\eta_j$ has at least one $\xi_i$ and at least one $\zeta_{i'}$ amongst its arguments.
\item\label{eq:com.components.13} The $\eta_j$ are all distinct as formal commutators in the $\xi_i,\zeta_{i'}$.
\item\label{eq:com.components.14} No $\eta_j$ has a component of the form $[\xi_i,\xi_{i'}]$ or $[\zeta_i,\zeta_{i'}]$.
\end{enumerate}
Lemma~\ref{lem:comm.decomp} also implies that every $\eta_j$ that is not of the form $[\xi_i,\zeta_{i'}]$ has weight greater than 2, and so, viewing the $\eta_i$ now as commutators in the $x^{(j)}_i$ and $y^{(j)}_i$, it follows from conditions (\ref{eq:com.components.6}), (\ref{eq:com.components.7}) and (\ref{eq:com.components.12}) that they satisfy condition (v) of the proposition. Moreover, condition (\ref{eq:com.components.10}) implies that they satisfy condition (i) of the proposition. Condition (ii) follows from conditions (\ref{eq:com.components.4}), (\ref{eq:com.components.5}) and (\ref{eq:com.components.11}). Condition (iii) follows from conditions (\ref{eq:com.components.6}), (\ref{eq:com.components.7}) and (\ref{eq:com.components.12}). Finally, in light of conditions (\ref{eq:com.components.8}), (\ref{eq:com.components.9}), (\ref{eq:com.components.13}) and (\ref{eq:com.components.14}), condition (iv) follows from Lemma~\ref{lem:higher.comms.different}.
\end{proof}
\begin{corollary}\label{cor:com.powers}
Let $\alpha$ be a commutator form of weight $r$, let $x_1,\ldots,x_r$ be elements of an $s$-step nilpotent group $\Gamma$, and let $(l_1,\ldots,l_r)$ be a vector of positive integers. Let $\xi_1,\ldots,\xi_k$ be the complete list, in ascending order with respect to $\prec$, of commutators in the $x_i$ of weight greater than $r$. Then $\alpha(x_1^{l_1},\ldots,x_r^{l_r})$ is equal in $\Gamma$ to a product, in some order, of precisely $\prod_il_i$ copies of $\alpha(x_1,\ldots,x_r)$ and at most $l^{\chi(\xi_j)}$ copies of each $\xi_j$.
\end{corollary}
\begin{proof}
The fact that $\alpha(x_1^{l_1},\ldots,x_r^{l_r})$ may be written as a product of commutators that includes at least $\prod_il_i$ copies of $\alpha(x_1,\ldots,x_r)$ follows from conditions (i) and (ii) of Proposition~\ref{prop:com.components}, and the fact that we may take the other commutators in this product to have weight greater than $r$ follows from condition (v) of that proposition. The bound on the number of copies of each $\xi_j$ and the fact that no further copies of $\alpha(x_1,\ldots,x_r)$ are required follow from condition (iv).
\end{proof}
\begin{proof}[Proof of Proposition~\ref{prop:com.powers}]
This follows from applying Corollary~\ref{cor:basic.collected} to the product from Corollary~\ref{cor:com.powers}.
\end{proof}
\section{Rough equivalence of progressions}\label{sec:rough.equiv.prog}
The purpose of this appendix is to prove the following result.
\begin{prop}\label{prop:prim.prog.to.nilp.prog}
Let $G$ be an $s$-step nilpotent group and suppose that $x_1,\ldots,x_r\in G$ and $L_1,\ldots,L_r\in\N$. Then
\[
P_\textup{ord}(x_1,\ldots,x_r;L)\subset P^*(x_1,\ldots,x_r;L)\subset P(x_1,\ldots,x_r;L)\subset P_\textup{ord}(x_1,\ldots,x_r;L)^{r^{O_s(1)}}
\]
\end{prop}
To facilitate the proof we introduce some notation. Given group elements $x_1,\ldots,x_k$ and a vector of positive integers $L=(L_1,\ldots,L_k)$ we define
\[
B(x_1,\ldots,x_k;L):=\bigcup_{i=1}^k\{l_ix_i:|l_i|\le L_i\}.
\]
We note the trivial inclusion
\begin{equation}\label{eq:prim.in.b^r}
B(x_1,\ldots,x_k;L)\subset P_\textup{ord}(x_1,\ldots,x_k;L),
\end{equation}
and that the role of $B(x_1,\ldots,x_k;L)$ in our argument will be comparable to that of the set $\mathfrak{b}$ appearing in the proof of \cite[Proposition 4.4]{tor.free.nilp}.
\begin{lemma}\label{lem:com.powers}
Let $x_1,\ldots,x_r$ be elements of an $s$-step nilpotent group and let $L=(L_1,\ldots,L_r)$ be a vector of positive integers. Let $\alpha$ be a commutator form of weight $r$. Then for every integer $m$ satisfying $|m|\le L_1\cdots L_r$ we have
\[
\alpha(x_1,\ldots,x_r)^m\in B(x_1,\ldots,x_r;L)^{O_s(1)}.
\]
\end{lemma}
\begin{proof}
For $r>s$ the commutator $\alpha(x_1,\ldots,x_r)$ is equal to the identity and therefore the lemma holds trivially, and so we may assume that
\begin{equation}\label{eq:com.powers.1}
r\le s.
\end{equation}
We shall in fact show that
\begin{equation}\label{eq:com.powers.0}
\alpha(x_1,\ldots,x_r)^m\in B(x_1,\ldots,x_r;L)^{O_{r,s}(1)};
\end{equation}
(\ref{eq:com.powers.1}) shows that this is sufficient to imply the lemma.

We may assume by induction that (\ref{eq:com.powers.0}) holds holds for all commutators of weight greater than $r$. By symmetry of $B(x_1,\ldots,x_r;L)$ we may assume that $m\ge0$, and it then follows directly from \cite[Lemma 4.3]{tor.free.nilp} that there are non-negative integers
\begin{equation}\label{eq:com.powers.1.5}
l_i^{(j)}\le L_i
\end{equation}
for $i=1,\ldots,r$ and $j=1,\ldots,2^{r-1}$ such that
\[
m=(l_1^{(1)}\cdots l_r^{(1)})+\ldots+(l_1^{(2^{r-1})}\cdots l_r^{(2^{r-1})}).
\]
Hence
\begin{equation}\label{eq:com.powers.2}
\alpha(x_1,\ldots,x_r)^m=\alpha(x_1,\ldots,x_r)^{l_1^{(1)}\cdots l_r^{(1)}}\cdots\alpha(x_1,\ldots,x_r)^{l_1^{(2^{r-1})}\cdots l_r^{(2^{r-1})}}.
\end{equation}
Let $\zeta_1,\ldots,\zeta_t$ be as defined in the statement of Proposition~\ref{prop:com.powers}, and note that
\begin{equation}\label{eq:com.powers.3}
t\ll_{r,s}1.
\end{equation}
It then follows from Proposition~\ref{prop:com.powers} that there exist, for $i=1,\ldots,t$ and $j=1,\ldots,2^{r-1}$, integers $m_i^{(j)}$ with
\begin{equation}\label{eq:com.powers.4}
|m_i^{(j)}|\le(l^{(j)})^{\chi(\eta_i)}
\end{equation}
such that
\[
\alpha(x_1,\ldots,x_r)^{l_1^{(j)}\cdots l_r^{(j)}}=\alpha(x_1^{l_1^{(j)}},\ldots,x_r^{l_r^{(j)}})\zeta_t^{m_t^{(j)}}\cdots\zeta_1^{m_1^{(j)}}.
\]
Substituting this into (\ref{eq:com.powers.2}) we may conclude that
\begin{equation}\label{eq:com.powers.5}
\alpha(x_1,\ldots,x_r)^m=\prod_{j=1}^{2^{r-1}}\alpha(x_1^{l_1^{(j)}},\ldots,x_r^{l_r^{(j)}})\zeta_t^{m_t^{(j)}}\cdots\zeta_1^{m_1^{(j)}}.
\end{equation}
Now (\ref{eq:com.powers.1.5}) implies that
\begin{equation}\label{eq:com.powers.6}
\alpha(x_1^{l_1^{(j)}},\ldots,x_r^{l_r^{(j)}})\in B(x_1,\ldots,x_r;L)^{O_r(1)},
\end{equation}
and by (\ref{eq:com.powers.4}) and induction we have
\begin{equation}\label{eq:com.powers.7}
\zeta_i^{m_i^{(j)}}\in B(x_1,\ldots,x_r;L)^{O_{r,s}(1)}.
\end{equation}
A combination of (\ref{eq:com.powers.3}), (\ref{eq:com.powers.5}), (\ref{eq:com.powers.6}) and (\ref{eq:com.powers.7}) gives (\ref{eq:com.powers.0}), as claimed.
\end{proof}
\begin{proof}[Proof of Proposition \ref{prop:prim.prog.to.nilp.prog}]
The first inclusion is trivial and the second is given by Corollary~\ref{cor:nilprog.in.nilp.prog}, and so it remains only to prove the third inclusion. Write $c_1,\ldots,c_t$ for the list of basic commutators in the $x_i$. The number of commutator shapes of weight at most $s$ is clearly finite, and there are at most $r^s$ commutators in $x_1,\ldots,x_r$ of any such shape, and so in particular we have
\begin{equation}\label{eq:pf.prim.prog.to.nilp.prog.0}
t\le r^{O_s(1)}.
\end{equation}
It is clear from the definition of $B$ that
\begin{equation}\label{eq:pf.prim.prog.to.nilp.prog.1}
B(x_{j_1},\ldots,x_{j_k};(L_{j_1},\ldots,L_{j_k}))\subset B(x_1,\ldots,x_r;L),
\end{equation}
Moreover, for a given basic commutator $c_i=\alpha(x_{j_1},\ldots,x_{j_k})$ it follows from Lemma~\ref{lem:com.powers} that whenever $|m_i|\le L^{\chi(c_i)}$ we have
\[
c_i^{m_i}\in B(x_{j_1},\ldots,x_{j_k};(L_{j_1},\ldots,L_{j_k}))^{O_s(1)},
\]
which by (\ref{eq:prim.in.b^r}) and (\ref{eq:pf.prim.prog.to.nilp.prog.1}) implies that
\[
c_i^{m_i}\in P_\textup{ord}(x_1,\ldots,x_r;L)^{O_s(1)}.
\]
However, the elements of $P(x_1,\ldots,x_k;L)$ are precisely those of the form
\[
c_1^{m_i}\cdots c_t^{m_t}
\]
with $|m_i|\le L^{\chi(c_i)}$, and so (\ref{eq:pf.prim.prog.to.nilp.prog.0}) implies the result.
\end{proof}
\section{Generating sets of $p$-groups}\label{sec:p.groups}
In this appendix we derive the following basic properties of $p$-groups.
\begin{lemma}\label{lem:ab.subgroup.rank}
If $\Gamma$ is an abelian $p$-group and $\Gamma'$ is a subgroup of $\Gamma$ then the rank of $\Gamma'$ is at most the rank of $\Gamma$.
\end{lemma}
\begin{lemma}\label{lem:burnside.basis}
Let $\Gamma$ be a $p$-group of rank $r$ and suppose that $S$ is a generating set for $\Gamma$. Then there is a subset $S'\subset S$ of cardinality $r$ that also generates $\Gamma$.
\end{lemma}
\begin{remark}Lemma \ref{lem:ab.subgroup.rank} still holds without the assumption that $\Gamma$ is a $p$-group. The corresponding statement for an arbitrary finite abelian group can be deduced from the $p$-group statement using arguments similar to those found in Section \ref{sec:img.of.multi.hom}. We do not need the more general statement in the present work, however, and so we omit the details. The assumption that $\Gamma$ is abelian, on the other hand, is necessary. For example, if $F$ is the free product of $n$ copies of the cyclic group with two elements and $\Gamma$ is the quotient $F/[F,[F,F]]$ then $\Gamma$ is of rank $n$ but $[\Gamma,\Gamma]$, for example, is of rank $n(n-1)/2$.
\end{remark}
\begin{remark}\label{rem:burnside.basis}Lemma \ref{lem:burnside.basis} does not necessarily hold if $\Gamma$ is not a $p$-group. For example, the set $\{2,3\}$ is a generating set for $\Z/6\Z$, but each of $\{2\}$ and $\{3\}$ generates a proper subgroup.
\end{remark}
The proof of Lemma~\ref{lem:ab.subgroup.rank} is straightforward. 
\begin{proof}[Proof of Lemma~\ref{lem:ab.subgroup.rank}]
\begin{sloppypar}Suppose that $\Gamma$ is of rank $r$, with minimal generating set $x_1,\ldots,x_r$, say. We will show by induction on $r$ that the rank of $\Gamma'$ is at most $r$, this being trivial when $r=0$.\end{sloppypar}

Let $m$ be the largest integer with the property that $\Gamma'<p^m\cdot\Gamma$. The subgroup $p^m\cdot\Gamma$ is generated by the set $p^mx_1,\ldots,p^mx_r$, and so has rank at most $r$. Therefore, upon replacing $\Gamma$ by $p^m\cdot\Gamma$ if necessary, we may assume that $\Gamma'$ is not contained in $p\cdot\Gamma$, and hence that there is some $y\in\Gamma'$ such that $px\ne y$ for every $x\in\Gamma$.

Since $x_1,\ldots,x_r$ generate $\Gamma$ we may write $y=l_1x_1+\ldots+l_rx_r$, and by definition of $y$ there must be at least one $i$ for which $l_i$ is not divisible by $p$. Without loss of generality we shall assume that $l_r$ is not divisible by $p$. Since the order of $x_r$ is a power of $p$, the integer $l_r$ has a multiplicative inverse, say $q$, modulo the order of the $x_r$. This implies that
\[
x_r=q(y-l_1x_1-\ldots-l_{r-1}x_{r-1}),
\]
and in particular that the quotient $\Gamma/\langle y\rangle$ is generated by $x_1+\langle y\rangle,\ldots,x_{r-1}+\langle y\rangle$, and hence has rank at most $r-1$. The group $\Gamma'/\langle y\rangle$ is isomorphic to a subgroup of $\Gamma/\langle y\rangle$, and hence has rank at most $r-1$ by induction. It follows that $\Gamma'$ has rank at most $r$, as claimed.
\end{proof}
The proof of Lemma~\ref{lem:burnside.basis} rests on a deeper property of $p$-groups, the statement of which requires a definition.
\begin{definition}[Frattini subgroup]Let $\Gamma$ be a finite $p$-group. Then the \emph{Frattini subgroup} $\Phi(\Gamma)$ of $\Gamma$ is defined by $\Phi(\Gamma):=\Gamma^p[\Gamma,\Gamma]$, where here (unlike elsewhere in this work) the notation $\Gamma^p$ means the group $\langle\,g^p\,|\,g\in\Gamma\,\rangle$.
\end{definition}
\begin{theorem}[Burnside basis theorem {\cite[Theorem 4.8]{khu}}]\label{thm:burnside.basis}Let $\Gamma$ be a $p$-group and write $\Phi(\Gamma)$ for its Frattini subgroup. Then a subset $S$ of $\Gamma$ generates $\Gamma$ if and only if the image of $S$ in $\Gamma/\Phi(\Gamma)$ generates $\Gamma/\Phi(\Gamma)$ as a vector space over $\F_p$.
\end{theorem}
\begin{proof}[Proof of Lemma~\ref{lem:burnside.basis}]
The Burnside basis theorem (Theorem~\ref{thm:burnside.basis}) implies in particular that all minimal generating sets of $\Gamma$ have the same cardinality, from which Lemma~\ref{lem:burnside.basis} follows immediately.
\end{proof}

\end{document}